\documentclass[a4paper]{amsart}

\usepackage[top=1in, bottom=1.25in, left=1.2in, right=1.2in]{geometry}
\usepackage{graphicx}
\usepackage[width=.9\textwidth]{caption}
\usepackage{amsthm, amssymb, amsmath, amsfonts, mathrsfs}
\usepackage{comment}

\usepackage[ocgcolorlinks, colorlinks=true, pdfstartview=FitV, linkcolor=blue, citecolor=blue, urlcolor=blue, pagebackref=false]{hyperref}

\usepackage{microtype}

\usepackage{xcolor}

\newtheorem{thm}{Theorem}[section]

\newtheorem{lem}[thm]{Lemma}
\newtheorem{cor}[thm]{Corollary}
\theoremstyle{remark}
\newtheorem{rem}[thm]{Remark}

\def\les{\lesssim}

\renewcommand{\leq}{\leqslant}
\renewcommand{\geq}{\geqslant}
\renewcommand{\subset}{\subseteq}

\newcommand{\E}{\mathbb{E}}
\newcommand{\EE}{\mathbf{E}^{0}}
\newcommand{\EEo}{\mathbf{E}^{\omega,\eps}}
\renewcommand{\P}{\mathbb{P}}
\newcommand{\PP}{\mathbf{P}^{0}}
\newcommand{\PPo}{\mathbf{P}^{\omega,\eps}}

\newcommand{\Ione}{\mathcal{I}}
\newcommand{\Itwo}{\mathcal{I}'}

\newcommand{\U}{\mathcal{U}}
\newcommand{\ZZ}{\mathcal{Z}}

\newcommand{\V}{\mathcal{V}}

\newcommand{\1}{\mathbf{1}}
\newcommand{\R}{\mathbb{R}}
\newcommand{\C}{\mathcal{C}}
\newcommand{\I}{\mathcal{I}}

\newcommand{\T}{\mathbb{T}}

\newcommand{\Z}{\mathbb{Z}}
\newcommand{\cP}{\mathscr{P}}
\newcommand{\ccP}{\mathcal{P}}

\newcommand{\eps}{\varepsilon}

\newcommand{\cA}{\mathcal{A}}

\numberwithin{equation}{section}

\begin{document}
\author[I. Corwin]{Ivan Corwin}
\address{I. Corwin, Columbia University,
Department of Mathematics,
2990 Broadway,
New York, NY 10027, USA,
and Clay Mathematics Institute, 10 Memorial Blvd. Suite 902, Providence, RI 02903, USA}
\email{ivan.corwin@gmail.com}

\title[KPZ equation and large deviation for RWRE]{Kardar-Parisi-Zhang equation and large deviations for random walks in weak random environments}
\author[Y. Gu]{Yu Gu}
\address{Y. Gu, Department of Mathematics, Building 380, Stanford University, Stanford, CA, 94305, USA}
\email{yg@math.stanford.edu}

\begin{abstract}
We consider the transition probabilities for random walks in $1+1$ dimensional space-time random environments (RWRE). For critically tuned weak disorder we prove a sharp large deviation result: after appropriate rescaling, the transition probabilities for the RWRE evaluated in the large deviation regime, converge to the solution to the stochastic heat equation (SHE) with multiplicative noise (the logarithm of which is the KPZ equation). We apply this to the exactly solvable Beta RWRE and additionally present a formal derivation of the convergence of  certain moment formulas for that model to those for the SHE.
\end{abstract}

\maketitle
%
%
%
%
%
%%%%%%%%%%%%%%%%%%%%%%%%%%%%%%%%%%%%%%%%%%%%%%%%%%%%%%%%%%%%%%
%%%%%%%%%%%%%%%%%%%%%%%%%%%%%%%%%%%%%%%%%%%%%%%%%%%%%%%%%%%%%%
%
%
%
%

\section{Introduction}
We consider random walks in $1+1$ dimensional space-time random environments (RWRE). The environment is specified by a sequence of zero-mean i.i.d. random variables (whose probability measure is denoted $\P$ and expectation operator $\E$):
\[
\omega=\{\omega_{i,j}:i\in \Z_{\geq 0},j\in \Z\},
\]
and a parameter $\eps>0$. For each realization of $\omega$, and each $\eps$ we consider a measure $\PPo$ on one-dimensional nearest neighbor walks  $\{S_n\}_{n\geq 0}$ started at the origin ($S_0=0$) which jump up or down according to the probabilities
\begin{equation}
\PPo(S_{n+1}=S_n+1)=\frac{1+\eps^\frac12\omega_{n,S_n}}{2}, \ \ \PPo(S_{n+1}=S_n-1)=\frac{1-\eps^\frac12\omega_{n,S_n}}{2}.
\label{eq:rwre}
\end{equation}
We assume that the $\omega$ and $\eps$ are such that all probabilities lie in $[0,1]$. When $\eps=0$, this becomes the measure of a simple symmetric random walk (SSRW), and we distinguish this measure by writing it as $\PP$ (we  use $\EEo$ and $\EE$ to denote the respective expectation operators). The parameter $\eps$ allows us to tune the strength of the disorder around that of the SSRW. In this paper we will consider the transition probability $\PPo(S_N=y)$ as a random variable (it inherits the randomness of $\omega$). We are interested in the large deviation regime whereby $y=vN$ for some $v\in (0,1)$.

For $\omega$ and $\eps$ fixed, subject to some hypotheses on $\omega$, \cite{RASY13} proves that $N^{-1}\log \PPo(S_N=vN)$ has a limit (which is called the rate function). That rate function is written in terms of the Legendre transform of another $N\to\infty$ limit. In the special case when the random variables $\tfrac{1}{2}(1+\eps^\frac12\omega_{n,S_n})$ are distributed according to the $\mathrm{Beta}(\alpha,\beta)$ distribution, \cite{barraquand2015random} proves an explicit and simple formula for the rate function. Further, \cite{barraquand2015random} proves that after centering $\log \PPo(S_N\geq vN)$ by $N$ times the rate function, the result fluctuates like $N^{1/3}$ times a GUE Tracy-Widom random variable. (This result was only proved for certain ranges of parameters $\alpha,\beta,v$ though should hold in general away from law of large numbers velocity.) Note that the result of \cite{barraquand2015random} involved the tail probabilities instead of the transition probabilities. A similar result should hold in both cases -- see the recent non-rigorous physics work of \cite{LeDoussal-Thierry} regarding the transition probability fluctuations. The occurrence of cube-root fluctuations and the GUE Tracy-Widom distribution demonstrates a relation between the Beta RWRE and the KPZ universality class \cite{ICReview}.

To further elucidate the connection between the RWRE and KPZ universality class, and to motivate our main results, let us observe how the RWRE transition probability is equal to the partition function for a directed polymer model. (We utilize the notational conventions from \cite{barraquand2015random} to make comparison with those results easier.) For $N\in\Z_{\geq 0}$ and $x\in\Z_{>0}$, we consider paths starting from $(0,1)$ and ending at $(N,x)$, which are allowed to make right steps by adding $e_1$ and diagonal steps by adding $e_1+e_2$ (here $e_1=(1,0)$ and $e_2=(0,1)$ are standard basis vectors). The weight of a path is the product of the weights of each edge along the path, and the weight of the edge $e$ is defined by
\begin{equation}
\left\{ \begin{array}{ll}
B_{i,j} & \mbox{ if $e$ is the horizontal edge } (i-1,j)\to (i,j),\\
1-B_{i,j} & \mbox{ if $e$ is the diagonal edge } (i-1,j-1)\to (i,j),
\end{array}
\right.
\label{eq:weight}
\end{equation}
where $\{B_{i,j}\}$ is a sequence of i.i.d. random variables supported on $[0,1]$. The point-to-point polymer partition function $Z(N,x)$ then equals the sum over all paths from $(0,1)$ to $(N,x)$ of the above defined path weights, and it satisfies the recursion relation
\begin{equation}
Z(N,x)=Z(N-1,x)B_{N,x}+Z(N-1,x-1)(1-B_{N,x})
\label{eq:repa}
\end{equation}
with initial data $Z(0,x)=\1\{x=1\}$. This polymer has the special property that weights leading into a particular vertex always sum to $1$. Let us note that the above recursion for $Z(N,x)$ is a special case of a random averaging process (RAP) which satisfies
$$
Z(N,x)=\sum_{k: |k|\leq M} Z(N-1,x+k)B_{N,x}(k),
$$
where $B_{i,j}=\{B_{i,j}(k): -M\leq k\leq M,\, \sum_{k:|k|\leq M} B_{i,j}(k)=1\}$ are i.i.d. probability vectors. One may try to generalize some of the results we prove herein to the RAP, though we do not pursue that here.

The RWRE transition probability is related to this partition function via a time reversal. If we let
\begin{equation}
\PPo(S_{n+1}=S_n+1)=B_{n,S_n}, \ \ \PPo(S_{n+1}=S_n-1)=1-B_{n,S_n},
\label{eq:timere}
\end{equation}
(i.e., $\omega_{i,j}=2\eps^{-\frac12}(B_{i,j}-\tfrac12)$) then for fixed $N$ and $x$ we
have the equality in law
\begin{equation}
Z(N,x)=\PPo(S_{N}=N-2x+2).
\label{eq:rwpoly}
\end{equation}
A similar result holds for the RAP whereby the resulting RWRE may have jump sizes randomly sampled from $\{-M,\ldots,M\}$ according to the environment; see \cite{balazs2006random}.

Directed polymer partition functions (in fact their logarithms) are expected to show KPZ-class fluctuations under general choices of weights. This conjecture is far from proved, having only been demonstrated for certain exactly solvable models -- see the review \cite{corwin2014macdonald} for various recent references. The aforementioned result for the Beta RWRE (or equivalently polymer) thus fits into this conjecture.

For directed polymer models with i.i.d. weights on vertices (instead of edges) of the form $e^{\beta \omega}$, \cite{alberts2014intermediate} introduced {\it intermediate disorder} or {\it weak noise} scaling whereby as $N$ goes to infinity, $x$ scales like $N^{1/2}$ and $\beta$ like $N^{-1/4}$. Under that scaling, the partition function converges to a continuum partition function whereby a Brownian bridge moves through a space-time white noise potential and assigns a weight to each path given by the exponential of the integral of the white noise along the path. That partition function is equal, via the Feynman-Kac representation, to the solution to the stochastic heat equation (SHE) with multiplicative white noise
$$
\partial_t\ZZ(t,x) = \frac{1}{2} \partial_{xx} \ZZ(t,x) + \ZZ(t,x) \dot{W}(t,x)
$$
with delta initial data $\ZZ(0,x) =\delta_{x}$. See, for example, \cite{alberts2014intermediate} and references therein regarding the definition of this equation and white noise $\dot{W}$. For reference, note that the logarithm of the stochastic heat equation solves the KPZ equation, and the weak noise scaling is such that the KPZ equation remains invariant under it.

The results of \cite{alberts2014intermediate} was proved via the convergence of the discrete chaos series for the polymer partition function to that of the continuum Wiener chaos series for $\ZZ(t,x)$. Since we will make use of this chaos series, let us recall it here for delta initial data:
\begin{equation}
\ZZ(t,x) = p(t,x) + \sum_{k=1}^{\infty} \int_{\Delta_k(t)\times \R^k} \prod_{\ell=1}^{k+1} p(t_{\ell}-t_{\ell-1}, x_{\ell}-x_{\ell-1}) \prod_{\ell=1}^{k} W(dt_{\ell},dx_{\ell}),
\label{eq:chaosSHE}
\end{equation}
with
\begin{equation}
\Delta_k(t)=\big\{(t_1,\ldots,t_k):0< t_1< \ldots<  t_k< t\big\},
\label{eq:Deltak}
\end{equation}
$t_0=0$, $t_{k+1}=t$, $x_{0}=0$, $x_{k+1}=x$ and the standard heat kernel $p(t,x) = (2\pi t)^{-1/2} \exp\{-x^2/2t\}$. See \cite{alberts2014intermediate} for the definition of the multiple stochastic integrals against space-time white noise.

\subsection{Main result}
Inspired by the intermediate disorder regime of the polymer model explored in \cite{alberts2014intermediate}, we seek here to analyze $\PPo(S_N=y)$ when $N\sim t/\eps^2$ and $y\sim vt/\eps^2+x/\eps$ for some fixed speed $v\in (0,1)$ and $t>0,x\in\R$. Notice that this scaling corresponds with that used in the polymer setup since the noise is tuned by a factor of $\eps^{1/2}$ around its deterministic value $1/2$. In light of this similarity, we expect that after appropriate scaling, $\PPo(S_N=y)$ should converge to the solution (perhaps up to simple scaling of coordinates) of the SHE with delta initial data. That is exactly what we prove here.
%
%We expect to obtain in the limit the SHE with multiplicative white noise, the logarithm of which is the KPZ equation. To illustrate the choice of the scalings, recall that the strength of the random environment is of order $\eps^\frac12$. It turns out that by observing on  the spatial and temporal scales of order $\eps^{-1}$ and $\eps^{-2}$, the KPZ equation is invariant: if we consider the KPZ equation with a weak noise
%\[
%\partial_t u=\partial_{xx} u+|\partial_xu|^2+\eps^\frac12\dot{W}(t,x),
%\]
%then $u_\eps(t,x)=u(t/\eps^2,x/\eps)$ satisfies
%\[
%\partial_tu_\eps=\partial_{xx}u_\eps+|\partial_xu_\eps|^2+\frac{1}{\eps^{3/2}}\dot{W}(\frac{t}{\eps^2},\frac{x}{\eps}),
%\]
%which is again the KPZ equation since $\eps^{-\frac32}\dot{W}(t/\eps^2,x/\eps)$ has the same distribution as $\dot{W}(t,x)$.

\begin{figure}
\begin{center}
\includegraphics[scale=.8]{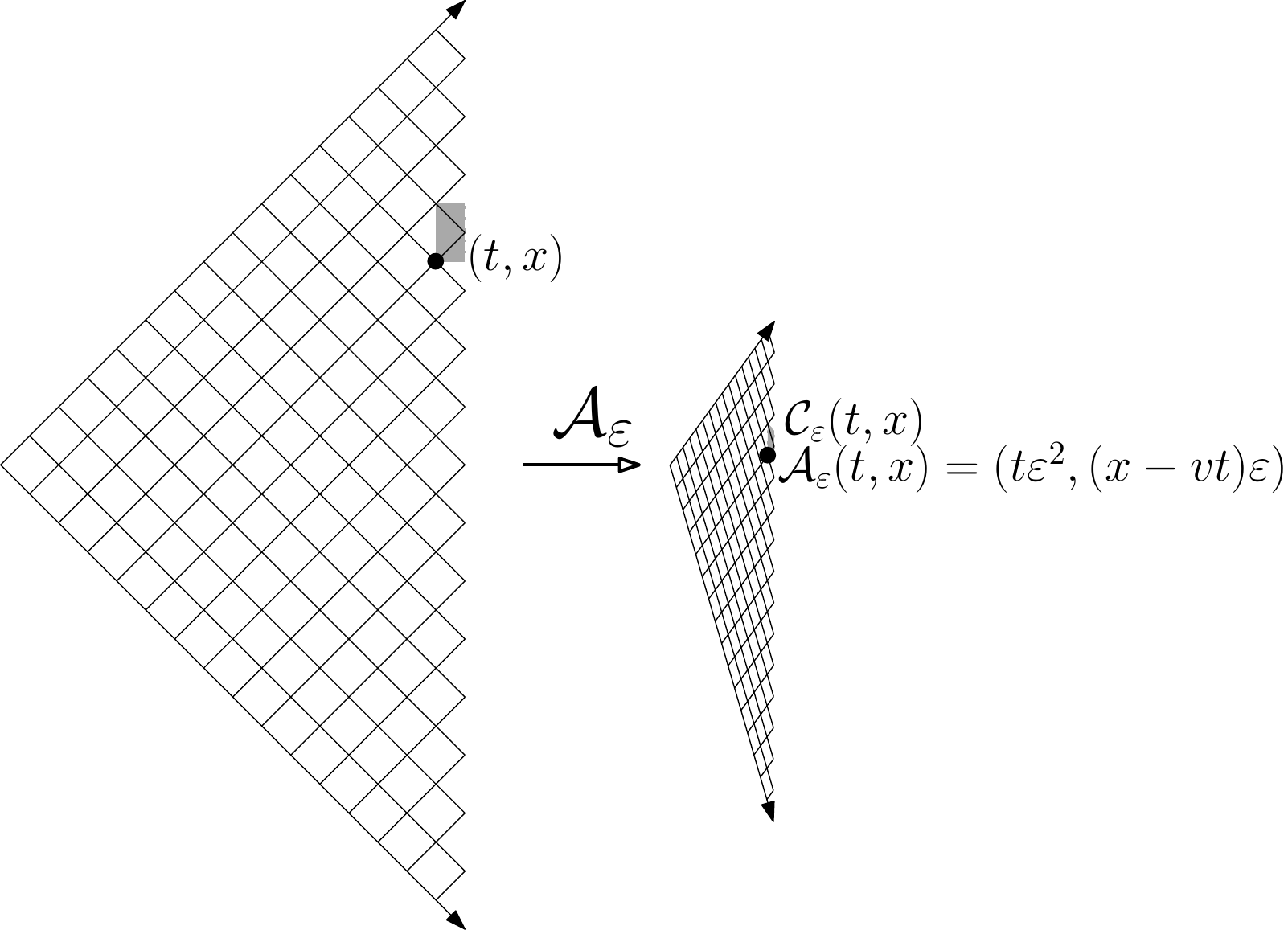}
\end{center}
\caption{On the left is the subspace of $\Z_2^2$ on which the RWRE trajectory lives. The image under the map $\cA_{\eps}$ is shown on the right. The grey box is the cell $\C_\eps(t,x)$ indexed by $(t,x)$.}
\label{fig:grid}
\end{figure}

Before stating our theorem, we introduce notation which identifies points in $\R_{+}\times\R$ with nearby points in the $\eps$-scaled lattice on which the random walk lives. This notation will also be used in the proofs.

The RWRE paths lie along the nonnegative time, even sublattice of $\Z^2$ which we denote by
$$
\Z_{2}^2=\{(t,x): t\in \Z_{\geq 0}, t+x\mbox{ is even}\}.
$$
For $v$ fixed and $\eps>0$ define the affine transformation $\cA_{\eps}:\R^2\to \R^2$ by
$$
\cA_{\eps}(t,x) = \big(t\eps^2,(x-vt)\eps\big),
$$
and denote the image of $\Z_2^2$ under $\cA_{\eps}$ by $\T_{\eps}$. For $(t,x)\in \Z_2^2$ let $\C_\eps(t,x)$ be the image under $\cA_{\eps}$ of the rectangle with vertices $(t,x), (t+1,x), (t+1,x+2), (t,x+2)$, which is a parallelgram. For convention we assume that the bottom and left edges of the rectangle are included whereas the top and right edges are not. As $(t,x)\in\Z_2^2$ varies, the $\C_\eps(t,x)$ form a disjoint partition of $\R_{+}\times \R$ into cells indexed by the bottom left point $(t,x)$ -- we refer to $\C_{\eps}$ as a tessellation and the $\C_{\eps}(t,x)$ as its cells. See Figure \ref{fig:grid} for an illustration of these definitions. For any function which is defined on $\Z_2^2$ (e.g. the transition probability) we may extend its domain to $\R_{+}\times \R$ by the convention that the value on each cell is constant and equal to the value at its bottom left corner. Also, for any $(t,x)$, we define $(t_{\eps},x_{\eps})$ to be the unique point in $\T_{\eps}$ which lies in the same cell as $(t,x)$. By these conventions, we have that the probability $\PPo$ in the left-hand side of \eqref{eq:LDPprob} remains unchanged as $(t,x)$ varying within the cell associated with $(t_{\eps},x_{\eps})$.

We are now prepared to state our main theorem.
%For any $t>0, x\in\R$, there exists a unique $(t_\eps,x_\eps)\in \T_\eps$ such that $(t,x)\in \C_\eps(t_\eps,x_\eps)$, and we choose $N=t_\eps/\eps^2, y=vt_\eps/\eps^2+x_\eps/\eps$ from now on. The following is our main result on the sharp large deviation behavior of $\PPo(S_N=y)$.
\begin{thm}
\label{t:mainTh}
Fix any distribution $\P$ on $\omega$ so that for $\eps$ small, the probabilities in \eqref{eq:rwre} are in $[0,1]$ and let $\sigma =\sqrt{2\E\{\omega_{i,j}^2\}}$. Then, for fixed $t>0, x\in\R$ and $v\in(0,1)$, we have
\begin{equation}
\frac{1}{\eps}e^{\frac{t_{\eps}}{\eps^2}\Ione(v)+\frac{x_{\eps}}{\eps}\Itwo(v)}\PPo\left(S_{\frac{t_{\eps}}{\eps^2}}= v\frac{t_{\eps}}{\eps^2}+\frac{x_{\eps}}{\eps}\right)\Rightarrow \U(t,x)
\label{eq:LDPprob}
\end{equation}
in distribution (with respect to the measure $\P$ on the $\omega$) as $\eps\to0$, where $\U(t,x)$ solves the following SHE (with scaled coefficients)
\begin{equation}
\partial_t \U(t,x)=\frac12(1-v^2)\,\partial_{xx} \U(t,x)+v\sigma\, \U(t,x)\,\dot{W}(t,x), \ \ \U(0,x)=2\delta(x),
\label{eq:spde}
\end{equation}
and
\begin{equation}
\Ione(v)=\frac{1-v}{2}\log (1-v)+\frac{1+v}{2}\log(1+v).
\end{equation}
%and
%\[
% \sigma^2=\frac{1-c_1^2}{c_2^2}, \ \ \tilde{\sigma}^2=\frac{2c_1^2}{c_2}\E\{\omega_{i,j}^2\}.
%\]
\end{thm}

Theorem~\ref{t:mainTh} implies that
\[
\log \PPo\left(S_{\frac{t_{\eps}}{\eps^2}}= v\frac{t_{\eps}}{\eps^2}+\frac{x_{\eps}}{\eps}\right)+\frac{t_{\eps}}{\eps^2} \Ione(v)+\frac{x_{\eps}}{\eps}\Itwo(v)-\log\eps \Rightarrow \log \U(t,x)
\]
in distribution, i.e., the recentered logarithm of the transition probability converges in distribution to the Hopf-Cole solution to the KPZ equation with narrow wedge initial data \cite{ACQ}.

\begin{rem}
For a SSRW starting from the origin (i.e., the $\eps=0$ case of the RWRE), a slight refinement of Cramer's theorem (see, for example, Lemma~\ref{l:ldprw}) shows that
\[
\frac{1}{\eps}e^{\frac{t}{\eps^2}\Ione(v)+\frac{x}{\eps}\Itwo(v)}\PP\left(S_{\frac{t}{\eps^2}}= v\frac{t}{\eps^2}+\frac{x}{\eps}\right)\to U(t,x)%=2p_{1-v^2}(t,x)
\]
as $\eps\to 0$, where %$p_{\sigma^2}(t,x)$ is the density of $N(0,\sigma^2 t)$, so
$U$ solves the heat equation (with scaled coefficient)
\begin{equation}
\partial_t U(t,x)=\frac12(1-v^2)\partial_{xx} U(t,x), \ \ U(0,x)=2\delta(x).
\label{eq:pde}
\end{equation}
Comparing \eqref{eq:spde} with \eqref{eq:pde}, we observe that the effect of the weak random environment on the transition probability is manifested as a multiplicative space-time white noise in the limiting equation.
\end{rem}

In light of the connection \eqref{eq:rwpoly} between the RWRE and directed polymers, the proof of Theorem~\ref{t:mainTh} also implies a similar convergence result for the polymer partition function $Z(N,x)$. The Beta polymer \cite{barraquand2015random} corresponds with taking $B_{i,j}$ i.i.d. as $\mathrm{Beta}(\alpha,\beta)$ random variables with $\alpha,\beta>0$. By the relation $\omega_{i,j}=2\eps^{-\frac12}(B_{i,j}-\tfrac12)$, fixing the distribution of the $B_{i,j}$ corresponds with taking $\omega_{i,j}$ dependent upon $\eps$. If we tune the parameters of the $\mathrm{Beta}(\alpha,\beta)$ random variables so that $\alpha=\beta=\eps^{-1}$, then the resulting $\omega_{i,j}$ has a non-trivial $\eps\to 0$ limit and the same method of proof as for Theorem~\ref{t:mainTh} applies and yields the following.
%\gucomment{need to emphasize that our random variables $\omega_{i,j}$ doesn't have to be bounded}

\begin{thm}
For the Beta polymer, fix $t>0,x\in\R$ and $\gamma\in(0,\frac12)$. Let $\alpha=\beta=\eps^{-1}$, we have
\[
\frac{1}{\eps}e^{\frac{t_{\eps}}{\eps^2}\Ione(1-2\gamma)-\frac{2x_{\eps}}{\eps}\Itwo(1-2\gamma)}Z\left(\frac{t_{\eps}}{\eps^2},\gamma\frac{t_{\eps}}{\eps^2}+\frac{x_{\eps}}{\eps}\right)\Rightarrow \V(t,x)%\frac{\gamma}{1-\gamma}\U(t,2x)
\]
in distribution as $\eps\to0$, where $\V(t,x)$ solves
\begin{equation}
\partial_t \V(t,x)=\frac12\gamma(1-\gamma)\,\partial_{xx}\V(t,x)+\frac{1-2\gamma}{\sqrt{2}}\,\V(t,x)\,\dot{W}(t,x), \ \ \V(0,x)=\frac{\gamma}{1-\gamma}\delta(x).
\label{eq:spde1}
\end{equation}
\label{c:poly}
\end{thm}
\begin{rem}
By \eqref{eq:rwpoly}, we have
\[
Z\left(\frac{t_{\eps}}{\eps^2},\gamma\frac{t_{\eps}}{\eps^2}+\frac{x_{\eps}}{\eps}\right)=\PPo\left(S_{\frac{t_{\eps}}{\eps^2}}=(1-2\gamma)\frac{t_{\eps}}{\eps^2}-\frac{2x_{\eps}}{\eps}+2\right),
\]
thus the tessellation, and consequently the meaning of $(t_{\eps},x_{\eps})$, in Theorem~\ref{c:poly} should be defined with $v=1-2\gamma$.
\end{rem}

As an independent confirmation of this result, in Section \ref{s:beta} we demonstrate how known formulas for moments of the Beta polymer partition function converge in the above scaling to those of the limiting SHE. Note that our weak convergence result does not imply convergence of moments (though it is reasonable to expect that such a stronger form of convergence may hold).

\vskip 3mm
Let us briefly sketch our approach to prove Theorem~\ref{t:mainTh}. The transition probability $\PPo(S_N=y)$ is the sum of the probabilities of all directed paths connecting $(0,0)$ to $(N,y)$ in $\Z_2^2$. These probabilities are products of the terms on the right-hand side of \eqref{eq:rwre} and can be expanded into powers of $\eps$. Indexing these sums in terms of the degree of the power of $\eps$ yields what is called a ``polynomial chaos" in the $\omega$ noise; see Lemma~\ref{l:pro}. The solution to the SHE also admits an expansion in term of Wiener chaos \eqref{eq:chaosSHE} and the proof then reduces to showing convergence of the polynomial to Wiener chaos series. For this, we apply the framework developed in \cite{caravenna2013polynomial} (which generalizes the results for polymers studied in \cite{alberts2014intermediate}) whereby a general criteria is given for such a convergence to hold. The key criteria to confirm is the $L^2$ convergence of the (deterministic) coefficients of each polynomial chaos to the corresponding coefficients of the Wiener chaos; see Lemma~\ref{l:conco}. Since the path of the RWRE coincides with that of the SSRW, the coefficients appearing in our setting are just joint transition probabilities of the SSRW and can be computed and bounded explicitly; see the sharp large deviation result for the SSRW presented in Appendix \ref{app:A}.
%The paper is organized as follows. In Section~\ref{s:pfth}, we present the proof of the main result

%In the following, we prove Theorem~\ref{t:mainTh} for random walk in general random environment. Then we consider the exactly solvable Beta RWRE/polymer and prove the convergence of moment formulas. %We use $a\les b$ when $a\leq Cb$ for some constant $C$ independent of $\eps,t,x$, and $C$ represents some constants that may change from line to line.

\subsection{Acknowledgements}
I.C. was partially supported by the NSF through DMS-1208998, the Clay Mathematics Institute through a Clay Research Fellowship, the Institute Henri Poincar\'{e} through the Poincar\'{e} Chair, and the Packard Foundation through a Packard Fellowship for Science and Engineering. Y.G. was partially supported by the NSF through DMS-1613301. We would like to thank the anonymous referees for several helpful suggestions.

\section{Proof of Theorems~\ref{t:mainTh} and \ref{c:poly}: convergence of the polynomial chaos}
\label{s:pfth}
The proof of Theorem~\ref{t:mainTh} (and similarly of Theorem~\ref{c:poly}) follows from the following two lemmas. After stating these lemmas, we complete the proof of the theorems and then devote the rest of this section to the proof of the lemmas. As in the introduction, we use $S_n$ to denote the trajectory of a nearest neighbor walk on $\Z$, $\PPo$ to denote the RWRE measure on such a walk given $\omega$ and $\eps$, and $\PP$ to denote the SSRW measure. Below, we also abbreviate points $(i,j)\in \Z_2^2$ by $z=(i,j)$, or $z_{l}=(i_l,j_l)$ when referring to multiple such points.

The following lemma expresses the random transition probability $\PPo(S_N=y)$ in terms of a polynomial chaos series with respect to the i.i.d. random variables $\omega$. It is easily proved by expanding the product formula for the transition probability in terms of powers of $\eps$.
\begin{lem}
\label{l:pro}
For any $(N,y)\in \Z_2^2$, we have
\begin{equation}
\PPo(S_N=y)=\PP(S_N=y)+\sum_{k=1}^N\eps^{\frac{k}{2}}\sum_{(z_1,\ldots,z_k)} \psi_k(z_1,\ldots,z_k)\omega_{z_1}\ldots \omega_{z_k},
\label{eq:Pold}
\end{equation}
where the summation $\sum_{(z_1,\ldots,z_k)}$ is over all possible $0\leq i_1<\ldots<i_k\leq N-1$ and $j_1,\ldots,j_k\in\Z$, and the expansion coefficients are given by
\begin{equation}
\begin{aligned}
\psi_k(z_1,\ldots,z_k)=\frac{1}{2^k}& \PP(S_{i_1}=j_1)\\
\times &\big[\PP(S_{i_2-i_1-1}=j_2-j_1-1)-\PP(S_{i_2-i_1-1}=j_2-j_1+1)\big]\\
\times  &\big[\PP(S_{i_3-i_2-1}=j_3-j_2-1)-\PP(S_{i_3-i_2-1}=j_3-j_2+1)\big]\\
\times &\ldots\\
\times &\big[\PP(S_{N-i_k-1}= y-j_k-1)-\PP(S_{N-i_k-1}= y-j_k+1)\big].
\end{aligned}
\label{eq:defpsi}
\end{equation}
\end{lem}

For $t>0$ and $x\in \R$ recall the conventions introduced before Theorem~\ref{t:mainTh} through which we associated $(t,x)$ with a pair $(t_{\eps},x_{\eps})\in \T_\eps$. We will study the RWRE transition probability when $N=t_\eps/\eps^2$ and $y=vt_\eps/\eps^2+x_\eps/\eps$. Towards this end, let us rewrite the chaos series for $\PPo(S_N=y)$ in terms of rescaled coordinates. For $z=(i,j)\in \Z_2^2$, denote $z_{\eps}= \cA_{\eps}(i,j)\in \T_{\eps}$ and for a sequence of $z_{\eps,1},\ldots,z_{\eps,k}\in\T_{\eps}$ define
\begin{equation}
\psi_{\eps,k}(z_{\eps,1},\ldots,z_{\eps,k})=\psi_k(\cA_{\eps}^{-1}z_{\eps,1},\ldots,\cA_\eps^{-1}z_{\eps,k})
\label{eq:defpsieps}
\end{equation}
with $\psi_k$ given by \eqref{eq:defpsi}. Also define a sequence of i.i.d. random variables indexed by $z_{\eps}\in \T_\eps$
by
$\tilde{\omega}_{z_\eps}=\omega_{\cA_\eps^{-1}z_\eps}$. With these rescaled coordinates, we have the following polynomial chaos series.

\begin{cor}
For fixed $t>0,x\in\R$, we have
\begin{equation}
\begin{aligned}
\PPo\left(S_{\frac{t_\eps}{\eps^2}}= v\frac{t_\eps}{\eps^2}+\frac{x_\eps}{\eps}\right)=&\PP\left(S_{\frac{t_\eps}{\eps^2}}= v\frac{t_\eps}{\eps^2}+\frac{x_\eps}{\eps}\right)\\
&+\sum_{k=1}^{t_\eps/\eps^2}\eps^{\frac{k}{2}} \sum_{(z_{\eps,1},\ldots,z_{\eps,k})\in \T_\eps^k}\psi_{\eps,k}(z_{\eps,1},\ldots,z_{\eps,k})\tilde{\omega}_{z_{\eps,1}}\ldots\tilde{\omega}_{z_{\eps,k}},
\end{aligned}
\label{eq:Pnew}
\end{equation}
with the summation $\sum_{(z_{\eps,1},\ldots,z_{\eps,k})\in \T_\eps^k}$ restricted to $0\leq i_{\eps,1}<\ldots<i_{\eps,k}\leq \frac{t_\eps}{\eps^2}-1$.
\end{cor}

It remains now to prove that as $\eps$ goes to zero, the above series converges to the SHE chaos series (after some minor rescaling of coordinates) given in \eqref{eq:chaosSHE}. Owing to the general machinery given in \cite[Section 2.3]{caravenna2013polynomial}, the main technical challenge in achieving this convergence is to prove the $L^2$-convergence of $\psi_{\eps,k}(z_{\eps,1},\ldots,z_{\eps,k})$ to the corresponding SHE chaos series coefficients. To state this result, we introduce the $L^2$-space for this convergence. Let $\Delta_k(t)$ be the time simplex defined in \eqref{eq:Deltak}.
As per our conventions, the domain of the function $\psi_{\eps,k}$ extends to all of $\Delta_k(t)\times \R^{k}$ by replacing each $(t_{l},x_l)$ by the corresponding $(t_{\eps,l},x_{\eps,l})$. If for two distinct choices of $l$, the resulting $t_{\eps,l}$ coincide, then set $\psi_{\eps,k}$ equal to zero. The function $\psi_{\eps,k}$ may also be extended to unordered times $t_1,\ldots,t_k\in \R_{+}$ by setting it equal to the corresponding value for the ordered times.
%\begin{equation}
%\psi_{\eps,k}((t_1,x_1),\ldots,(t_k,x_k))=\psi_{\eps,k}((t_{\eps,1},x_{\eps,1}),\ldots,(t_{\eps,k},x_{\eps,k})),
%\label{eq:defpsicon}
%\end{equation}
%where $(t_l,x_l)\in \C_\eps(t_{\eps,l},x_{\eps,l})$ with $(t_{\eps,l},x_{\eps,l})\in\T_\eps$, $l=1,\ldots,k$.

%If $t_{\eps,l_1}=t_{\eps,l_2}$ for some $l_1,l_2$, we define the l.h.s. of \eqref{eq:defpsicon} as zero. \note{this extension in the definition seems unnecessary since $\Delta_k(t)$ is defined with strict inequalities.}

Let $p_{\sigma^2}(t,x) = (2\pi \sigma^2 t)^{1/2} \exp\{-x^2/2\sigma^2t\}$ be the density of normal distribution $N(0,\sigma^2t)$.
\begin{lem}
\label{l:conco}
For fixed $t>0,x\in\R, k\geq 1$, let $t_0=0,x_0=0,t_{k+1}=t,x_{k+1}=x$. Then
\begin{equation}
\begin{aligned}
\frac{1}{\eps^{1+k}} e^{\frac{t_\eps}{\eps^2}\Ione(v)+\frac{x_\eps}{\eps}\Itwo(v)} \psi_{\eps,k}\big((t_1,x_1),\ldots,(t_k,x_k)\big)
\to 2(2v)^k\prod_{l=1}^{k+1} p_{1-v^2}(t_l-t_{l-1},x_l-x_{l-1})
\label{eq:conco}
\end{aligned}
\end{equation}
in $L^2(\Delta_k(t)\times\R^k)$. In addition,
\begin{equation}
\lim_{M\to\infty}\limsup_{\eps\to0}\sum_{k>M}\left\|\frac{1}{\eps^{1+k}} e^{\frac{t_\eps}{\eps^2}\Ione(v)}e^{\frac{x_\eps}{\eps}\Itwo(v)}\psi_{\eps,k}\right\|_{L^2(\Delta_k(t)\times \R^{k})}^2=0.
\label{eq:ubdl2}
\end{equation}
\end{lem}

%The $L^2$ convergence in Lemma~\ref{l:conco} comes from a pointwise convergence and a uniform estimate.
Now we can prove the main result.
\begin{proof}[Proof of Theorem~\ref{t:mainTh}] First observe that the left-hand side of \eqref{eq:LDPprob} is equal to $\frac{1}{\eps} e^{\frac{t_\eps}{\eps^2}\Ione(v)+\frac{x_\eps}{\eps}\Itwo(v)}$ times the polynomial chaos on the right-hand side of \eqref{eq:Pnew}.  We apply the general criteria given in \cite[Theorem 2.3]{caravenna2013polynomial} to prove convergence of a polynomial chaos to a Wiener chaos. There are three criteria which we must confirm to apply these results:
 %
%  to prove the convergence of
% \begin{equation}
% \begin{aligned}
% &\frac{1}{\eps} e^{\frac{t_\eps}{\eps^2}\I_1(v)+\frac{x_\eps}{\eps}\I_2(v)} \PPo\left(S_{\frac{t_\eps}{\eps^2}}= v\frac{t_\eps}{\eps^2}+\frac{x_\eps}{\eps}\right) \\
% =&\frac{1}{\eps} e^{\frac{t_\eps}{\eps^2}\I_1(v)+\frac{x_\eps}{\eps}\I_2(v)} \P\left(S_{\frac{t_\eps}{\eps^2}}= v\frac{t_\eps}{\eps^2}+\frac{x_\eps}{\eps}\right)\\
% &+\sum_{k=1}^{t_\eps/\eps^2} \sum_{(z_{\eps,1},\ldots,z_{\eps,k})\in \T_\eps^k}\frac{1}{\eps} e^{\frac{t_\eps}{\eps^2}\I_1(v)+\frac{x_\eps}{\eps}\I_2(v)}\eps^{\frac{k}{2}}\psi_{\eps,k}(z_{\eps,1},\ldots,z_{\eps,k})\tilde{\omega}_{z_{\eps,1}}\ldots\tilde{\omega}_{z_{\eps,k}},
% \end{aligned}
% \label{eq:pex}
% \end{equation}
% it suffices to check
\begin{enumerate}
\item[(i)] The random variables $\tilde\omega$ are i.i.d. with $\E\{\tilde\omega\}=0$ and $\E\{\tilde\omega^2\}<\infty$.

\item[(ii)] The volume of the cells in the tessellation $\C_\eps$ is $2\eps^3$. By Lemma~\ref{l:ldprw} and Lemma~\ref{l:conco}, we have,
 \[
 \frac{1}{\eps} e^{\frac{t_\eps}{\eps^2}\Ione(v)+\frac{x_\eps}{\eps}\Itwo(v)} \PP\left(S_{\frac{t_\eps}{\eps^2}}= v\frac{t_\eps}{\eps^2}+\frac{x_\eps}{\eps}\right)\to 2p_{1-v^2}(t,x)
 \]
 and, in $L^2(\Delta_k(t)\times\R^k)$,
 \[
 \begin{aligned}
 &\frac{1}{\eps} e^{\frac{t_\eps}{\eps^2}\Ione(v)+\frac{x_\eps}{\eps}\Itwo(v)}\eps^{\frac{k}{2}} \psi_{\eps,k}\big((t_1,x_1),\ldots,(t_k,x_k)\big)|\C_\eps|^{-\frac{k}{2}}\\
 &=\frac{1}{2^{\frac{k}{2}}}\frac{1}{\eps^{1+k}} e^{\frac{t_\eps}{\eps^2}\Ione(v)+\frac{x_\eps}{\eps}\Itwo(v)}\psi_{\eps,k}\big((t_1,x_1),\ldots,(t_k,x_k)\big)\\
 &\to 2(\sqrt{2}v)^k\prod_{l=1}^{k+1} p_{1-v^2}(t_l-t_{l-1},x_l-x_{l-1}).
\end{aligned}
 \]

\item[(iii)] By Lemma~\ref{l:conco}, the tail
  \[
  \lim_{M\to\infty}\limsup_{\eps\to0}\sum_{k>M}\sum_{(z_{\eps,1},\ldots,z_{\eps,k})\in \T_\eps^k}\left|\frac{1}{\eps} e^{\frac{t_\eps}{\eps^2}\Ione(v)+\frac{x_\eps}{\eps}\Itwo(v)}\eps^{\frac{k}{2}}\psi_{\eps,k}(z_{\eps,1},\ldots,z_{\eps,k})\right|^2=0.
  \]
\end{enumerate}

It then follows from \cite[Theorem 2.3]{caravenna2013polynomial} that
 \[
 \begin{aligned}
 &\frac{1}{\eps} e^{\frac{t_\eps}{\eps^2}\Ione(v)+\frac{x_\eps}{\eps}\Itwo(v)} \PPo\left(S_{\frac{t_\eps}{\eps^2}}= v\frac{t_\eps}{\eps^2}+\frac{x_\eps}{\eps}\right) \\
 \Rightarrow& 2p_{1-v^2}(t,x)+2\sum_{k=1}^\infty \int_{\Delta_k(t)\times \R^k}\prod_{l=1}^{k+1} p_{1-v^2}(t_l-t_{l-1},x_l-x_{l-1}) \prod_{l=1}^kv\sigma W(dt_l,dx_l)
 \end{aligned}
 \]
in distribution, with $\sigma=\sqrt{2\E\{\omega_{i,j}^2\}}$ and $W(dt,dx)$ a space-time white noise. To complete the proof, we only need to note that the limit is the series of Wiener chaos expansion of \eqref{eq:spde}.
 \end{proof}

\begin{proof}[Proof of Theorem~\ref{c:poly}] The results of Theorem~\ref{t:mainTh} hold in slightly more generality than stated. In particular, the distribution (denoted $\P$ and $\E$) of the $\omega$ may depend on $\eps$ (denoted $\P_{\eps}$ and $\E_{\eps}$) so long as $\E_{\eps}\{\omega_{i,j}\}=0$ and $2\E_{\eps}\{\omega_{i,j}^2\}\to \sigma^2$ as $\eps\to 0$. Then, an inspection of \cite[Theorem 2.3]{caravenna2013polynomial}  reveals that the same conclusion holds as in Theorem~\ref{t:mainTh}. With this in mind, define $\eps$-dependent $\omega_{i,j}=2\eps^{-1/2}\left(B_{i,j}-\frac12\right)$ with $B_{i,j}\sim \mathrm{Beta}(\eps^{-1},\eps^{-1})$, and note that $\E_{\eps}\{\omega_{i,j}\}=0$ and $\E_{\eps}\{\omega_{i,j}^2\}\to \frac12$ as $\eps\to0$. Thus, recalling that $Z(N,y)=\PPo(S_N=N-2y+2)$ and applying the above noted generalization of Theorem~\ref{t:mainTh}, we achieve the conclusion of Theorem~\ref{c:poly}.
\end{proof}
%In light of this, and the fact that $Z(N,y)=\PPo(S_N=N-2y+2)$ when
%
% \[
% \begin{aligned}
% &\frac{1}{\eps}e^{\frac{t_\eps}{\eps^2}\I_1(1-2\gamma)-\frac{2x_{\eps}}{\eps}\I_2(1-2\gamma)}Z\left(\frac{t_\eps}{\eps^2},\gamma\frac{t_\eps}{\eps^2}+\frac{x_{\eps}}{\eps}\right)\\
% =&\frac{1}{\eps}e^{\frac{t_\eps}{\eps^2}\I_1(1-2\gamma)-\frac{2x_{\eps}}{\eps}\I_2(1-2\gamma)}
%\PPo\left(S_{\frac{t_\eps}{\eps^2}}=(1-2\gamma)\frac{t_\eps}{\eps^2}-\frac{2x_\eps}{\eps}+2\right)\Rightarrow \V(t,x)
% \end{aligned}
% \]
% in distribution as $\eps\to0$.

\subsection{Proofs of Lemma~\ref{l:pro} and Lemma~\ref{l:conco}}

The proof of Lemma~\ref{l:pro} is a straightforward calculation using the properties of the SSRW. To prove Lemma~\ref{l:conco}, %we note that by \eqref{eq:defpsi} and \eqref{eq:defpsieps}, $\psi_{\eps,k}$ only depends on the distribution of $S_n$, and
 it is clear by \eqref{eq:defpsi} and \eqref{eq:defpsieps} that we need to analyze the sharp large deviation of the SSRW, i.e., the asymptotic behavior of
 \begin{equation}
\cP_\eps(t_\eps,x_\eps,m_1,m_2):= \PP\left(S_{\frac{t_\eps}{\eps^2}+m_1}=v\frac{t_{\eps}}{\eps^2}+\frac{x_\eps}{\eps}+m_2\right)
 \label{eq:defcp}
 \end{equation}
 for any $t>0,x\in\R$ and $(m_1,m_2)\in \{(0,0),(-1,-1),(-1,1)\}$. This is left to Appendix \ref{app:A}.

\begin{proof}[Proof of Lemma~\ref{l:pro}]
By the definition of RWRE it follows that
\begin{equation}
\begin{aligned}
\PPo(S_N=y)
%=\sum_{s_1,\ldots,s_N}\prod_{i=0}^{N-1}\left[\frac12+(s_{i+1}-s_i)\frac{\eps^\frac12}{2} \omega_{i,s_i}\right],
%\end{aligned}
%\end{equation}
%where the summation is over those paths of a SSRW connecting $0$ and $y$ in $N$ steps, with the trajectory denoted by $(0,0),(1,s_1),\ldots,(N,s_N)$. This can be rewritten as
%\begin{equation}
%\begin{aligned}
%\PPo(S_N=y)=\quad&\EE\left\{\prod_{i=0}^{N-1}\left[\frac12+(S_{i+1}-S_i)\frac{\eps^\frac12}{2} \omega_{i,S_i}\right]\1\{S_N=y\}\right\}2^N\\
=\quad&\EE\left\{\prod_{i=0}^{N-1}\left[1+(S_{i+1}-S_i)\eps^\frac12\omega_{i,S_i}\right]\1\{S_N=y\}\right\},
\end{aligned}
\label{eq:ldpr}
\end{equation}
where $\EE$ is with respect to the SSRW measure on $S$. %, the symmetric simple random walk. We emphasize that the $S_N$ on the l.h.s. of \eqref{eq:ldpr} is the RWRE and the $S$ on the r.h.. is the simple random walk.
Expanding the right-hand side of \eqref{eq:ldpr} we obtain (note below that the index of $S$ in $S_{i_l+1}$ is $i_{l}+1$)
\begin{equation*}
\begin{aligned}
\PPo(S_N=y)= \PP(S_N=y) +\sum_{k=1}^N\eps^\frac{k}{2} \sum_{0\leq i_1<\ldots<i_k\leq N-1}\EE\left\{\prod_{l=1}^k(S_{i_l+1}-S_{i_l})\omega_{i_l,S_{i_l}}\1\{S_N= y\}\right\}.
\end{aligned}
\end{equation*}
The expectation can be evaluated as
\begin{equation*}
\begin{aligned}
\EE\left\{\prod_{l=1}^k(S_{i_l+1}-S_{i_l})\omega_{i_l,S_{i_l}}\1\{S_N= y\}\right\}=\sum_{(z_1,\ldots,z_k)}\psi_k(z_1,\ldots,z_k)\omega_{z_1}\ldots\omega_{z_k},
\end{aligned}
\end{equation*}
where the summation $\sum_{(z_1,\ldots,z_k)}$ is over all possible $0\leq i_1<\ldots<i_k\leq N-1$ and $j_1,\ldots,j_k\in\Z$, and the coefficients
\begin{equation}
\begin{aligned}
&\psi_k(z_1,\ldots,z_k) = \sum_{\tau_1,\ldots,\tau_k=\pm1} \left(\prod_{l=1}^k \tau_l\right)\PP\left(\bigcap_{l=1}^k\{S_{i_l}=j_l,S_{i_l+1}=j_l+\tau_l\}\bigcap \{S_N=y\}\right).%\P(S_{i_1}=j_1,S_{i_1+1}=j_1+\tau_1,\ldots, S_{i_k}=j_k,S_{i_k+1}=j_k+\tau_k,S_N= y).
\end{aligned}
\label{eq:defpsinew}
\end{equation}
Using the independence of increments of the SSRW in \eqref{eq:defpsinew}, we arrive at the expression in \eqref{eq:defpsi} and thus complete the proof.
%\begin{equation}
%\begin{aligned}
%\psi_k(z_1,\ldots,z_k)=\frac{1}{2^k}& \PP(S_{i_1}=j_1)\\
%\times &[\PP(S_{i_2-i_1-1}=j_2-j_1-1)-\PP(S_{i_2-i_1-1}=j_2-j_1+1)]\\
%\times &\ldots\\
%\times  &[\PP(S_{i_k-i_{k-1}-1}=j_k-j_{k-1}-1)-\PP(S_{i_k-i_{k-1}-1}=j_k-j_{k-1}+1)]\\
%\times &[\PP(S_{N-i_k-1}= y-j_k-1)-\PP(S_{N-i_k-1}= y-j_k+1)].
%\end{aligned}
%\label{eq:psi1}
%\end{equation}
%The proof is complete.
\end{proof}

\begin{proof}[Proof of Lemma~\ref{l:conco}]
We first prove the convergence of \eqref{eq:conco} for fixed $0<t_1<\ldots<t_k<t$ and $x_1,\ldots,x_k\in\R$, then we prove the convergence is also in $L^2(\Delta_k(t)\times\R^k)$ for fixed $k$. In the end, we derive a bound on $\|\eps^{-k-1} e^{\frac{t_\eps}{\eps^2}\Ione(v)}e^{\frac{x_\eps}{\eps}\Itwo(v)}\psi_{\eps,k}\|_{L^2(\Delta_k(t)\times \R^{k})}$ for large $k$ to prove \eqref{eq:ubdl2}.

For $0<t_1<\ldots<t_k<t$ and $x_1,\ldots,x_k\in\R$, by \eqref{eq:defpsi}, \eqref{eq:defpsieps}, and \eqref{eq:defcp}, we have
\begin{equation}
\begin{aligned}
\frac{1}{\eps^{1+k}} e^{\frac{t_\eps}{\eps^2}\Ione(v)+\frac{x_\eps}{\eps}\Itwo(v)}\psi_{\eps,k}\big((t_1,x_1),\ldots,(t_k,x_k)\big)
= \frac{1}{2^k}  \ccP_\eps(t_1,x_1)\prod_{l=2}^{k+1} \ccP_\eps(t_l,x_l,t_{l-1},x_{l-1}),%\frac{1}{2^k} \cP_1(i_{\eps,1},j_{\eps,1})\left(\prod_{l=2}^k \cP_2(i_{\eps,l},j_{\eps,l},i_{\eps,l-1},j_{\eps,l-1})\right)\cP_3(i_{\eps,t},j_{\eps,x},i_{\eps,k},j_{\eps,k})
\end{aligned}
\end{equation}
where
\[
\ccP_\eps(t_1,x_1)=\frac{1}{\eps} e^{\frac{t_{\eps,1}}{\eps^2}\Ione(v)+\frac{x_{\eps,1}}{\eps}\Itwo(v)}\cP_\eps(t_{\eps,1},x_{\eps,1},0,0),
\]
and for $2\leq l\leq k+1$,
\[
\begin{aligned}
&\ccP_\eps(t_l,x_l,t_{l-1},x_{l-1})=\frac{1}{\eps}e^{\frac{t_{\eps,l}-t_{\eps,l-1}}{\eps^2}\Ione(v)+\frac{x_{\eps,l}-x_{\eps,l-1}}{\eps}\Itwo(v)}\\
&\times\big[\cP_\eps(t_{\eps,l}-t_{\eps,l-1},x_{\eps,l}-x_{\eps,l-1},-1,-1)-\cP_\eps(t_{\eps,l}-t_{\eps,l-1},x_{\eps,l}-x_{\eps,l-1},-1,1)\big].
%&\frac{1}{\eps}e^{\frac{t_{\eps,l}-t_{\eps,l-1}}{\eps^2}\I_1(v)+\frac{x_{\eps,l}-x_{\eps,l-1}}{\eps}\I_2(v)}\\
%\times &\left[\P\left(S_{\frac{t_{\eps,l}-t_{\eps,l-1}}{\eps^2}-1}=v\frac{t_{\eps,l}-t_{\eps,l-1}}{\eps^2}+\frac{x_{\eps,l}-x_{\eps,l-1}}{\eps}-1\right)\\
%&-\P\left(S_{\frac{t_{\eps,l}-t_{\eps,l-1}}{\eps^2}-1}=v\frac{t_{\eps,l}-t_{\eps,l-1}}{\eps^2}+\frac{x_{\eps,l}-x_{\eps,l-1}}{\eps}+1\right)\right].
\end{aligned}
\]

By Lemma~\ref{l:ldprw}, we have $\ccP_\eps(t_1,x_1)\to 2p_{1-v^2}(t_1,x_1)$, and for $2\leq l\leq k+1$ we have $\ccP_\eps(t_l,x_l,t_{l-1},x_{l-1}) \to 4vp_{1-v^2}(t_l-t_{l-1},x_l-x_{l-1})$. This implies the pointwise convergence
\begin{equation}
\begin{aligned}
\frac{1}{\eps^{k+1}} e^{\frac{t_\eps}{\eps^2}\Ione(v)}e^{\frac{x_\eps}{\eps}\Itwo(v)}\psi_{\eps,k}\big((t_1,x_1),\ldots,(t_k,x_k)\big)
\to
 2(2v)^k \prod_{l=1}^{k+1}p_{1-v^2}(t_l-t_{l-1},x_l-x_{l-1}).
\end{aligned}
\label{eq:conpp}
\end{equation}

To prove the convergence in $L^2(\Delta_k(t)\times\R^k)$, we recall the uniform bound \eqref{eq:uesrw} which holds for large enough values of $C$:
\begin{equation}
\begin{aligned}
\ccP_\eps(t_1,x_1)&\leq C\left(\frac{1}{\eps}\1\{t_1\leq 100\eps^2,|x_1|\leq C\eps\}+\1\{t_1>100\eps^2\}\frac{1}{\sqrt{t_1}}e^{-\frac{x_1^2}{Ct_1}}\right),\\
\ccP_\eps(t_l,x_l,t_{l-1},x_{l-1})
&\leq C\bigg(\frac{1}{\eps}\1\{|t_l-t_{l-1}|\leq 100\eps^2,|x_l-x_{l-1}|\leq C\eps\} \\
&\quad\quad\quad +\1\{|t_l-t_{l-1}|>100\eps^2\}\frac{1}{\sqrt{t_l-t_{l-1}}}e^{-\frac{(x_l-x_{l-1})^2}{C(t_l-t_{l-1})}}\bigg).
\end{aligned}
\label{eq:bdldp1}
\end{equation}
Letting
\[
\Delta_{k,\eps}(t)=\big\{(t_1,\ldots,t_k)\in \Delta_k(t): \min_{l=1,\ldots,k+1}|t_l-t_{l-1}|\leq 100\eps^2 \big\},
\]
we will show that the integral in $\Delta_{k,\eps}(t)\times \R^k$ is negligible as $\eps\to0$. Before doing so,
%and outside $\Delta_{k,\eps}(t)$ we have the Gaussian bounds in \eqref{eq:bdldp1}, so we can apply dominated convergence theorem together with the pointwise convergence in \eqref{eq:conpp} to pass to the limit.
observe that in $\Delta_k(t)\setminus\Delta_{k,\eps}(t)$, we have
\[
\begin{aligned}
\frac{1}{\eps^{1+k}} e^{\frac{t_\eps}{\eps^2}\Ione(v)+\frac{x_\eps}{\eps}\Itwo(v)}\psi_{\eps,k}\big((t_1,x_1),\ldots,(t_k,x_k)\big)
\leq C^k\prod_{l=1}^{k+1} \frac{1}{\sqrt{t_l-t_{l-1}}}e^{-\frac{(x_l-x_{l-1})^2}{C(t_l-t_{l-1})}}.
\end{aligned}
\]
The right-hand side above is in $L^2(\Delta_k(t)\times \R^k)$. Thus, by dominated convergence and the pointwise convergence in \eqref{eq:conpp} the following convergence holds in $L^2(\Delta_k(t)\times \R^k)$:
\begin{equation}
\begin{aligned}
\frac{1}{\eps^{1+k}} e^{\frac{t_\eps}{\eps^2}\Ione(v)+\frac{x_\eps}{\eps}\Itwo(v)}\psi_{\eps,k}\big((t_1,x_1),\ldots,(t_k,x_k)\big)\,\1\{\Delta_k(t)\setminus\Delta_{k,\eps}(t)\}%\prod_{l=1}^{k+1}1_{\{|t_l-t_{l-1}|>100\eps^2\}}\\
\to 2(2v)^k \prod_{l=1}^{k+1}p_{1-v^2}(t_l-t_{l-1},x_l-x_{l-1}).
\end{aligned}
\label{eq:dct}
\end{equation}
In addition (for our later proof of \eqref{eq:ubdl2}) we have
\begin{equation}
\begin{aligned}
&\left\|\frac{1}{\eps^{1+k}} e^{\frac{t_\eps}{\eps^2}\Ione(v)+\frac{x_\eps}{\eps}\Itwo(v)}\psi_{\eps,k}\,\1\{\Delta_k(t)\setminus\Delta_{k,\eps}(t)\}\right\|_{L^2(\Delta_k(t)\times\R^k)}^2\\
&\leq  C^k \int_{\Delta_k(t)\times \R^k} \prod_{l=1}^{k+1} \frac{1}{t_l-t_{l-1}}e^{-\frac{(x_l-x_{l-1})^2}{C(t_l-t_{l-1})}}d\mathbf{t}d\mathbf{x}\\
&\leq \frac{C^k}{\sqrt{t}}e^{-\frac{x^2}{Ct}}\int_{\Delta_k(t)} \prod_{l=1}^{k+1} \frac{1}{\sqrt{t_l-t_{l-1}}}d\mathbf{t}\\
&=\frac{C^k}{\sqrt{t}}e^{-\frac{x^2}{Ct}}t^{\frac{k-1}{2}}\int_{\Delta_k(1)} \frac{1}{\sqrt{t_1}}\frac{1}{\sqrt{t_2-t_1}}\ldots \frac{1}{\sqrt{1-t_k}}d\mathbf{t}\\
&\leq C^k e^{-\frac{x^2}{Ct}}t^{\frac{k}{2}-1} e^{-\frac{k\log k}{C}},
\end{aligned}
\label{eq:bd1}
\end{equation}
where the last step comes from the evaluation of the Dirichlet integral. It is clear that
\begin{equation}
\sum_{k=1}^\infty C^k e^{-\frac{x^2}{Ct}}t^{\frac{k}{2}-1} e^{-\frac{k\log k}{C}}<\infty.
\label{eq:bd11}
\end{equation}

Now we consider the integral in $\Delta_{k,\eps}(t)$ with the aim of showing that it is negligible as $\eps\to 0$. We change variables $t_l-t_{l-1}\mapsto \tau_l$, $x_l-x_{l-1}\mapsto y_l$ and use \eqref{eq:bdldp1} to derive that
\[
\begin{aligned}
&\left\|\frac{1}{\eps^{1+k}} e^{\frac{t_\eps}{\eps^2}\Ione(v)+\frac{x_\eps}{\eps}\Itwo(v)}\psi_{\eps,k}\right\|_{L^2(\Delta_{k,\eps}(t)\times\R^k)}^2\\
&\leq  C^k\int_{[0,\infty)^k}\int_{\R^k} \prod_{l=1}^{k+1} \left[\frac{1}{\eps^2} \1\{|\tau_l|\leq 100\eps^2, |y_l|\leq C\eps\}+\frac{1}{\tau_l}e^{-\frac{y_l^2}{C\tau_l}}\1\{|\tau_l|>100\eps^2\}\right]\\
&\qquad\qquad\qquad\qquad \1\Big\{\sum_{l=1}^{k+1}\tau_l=t,\min_l \tau_l\leq 100\eps^2\Big\}\1\Big\{\sum_{l=1}^{k+1} y_l=x\Big\} d\pmb{\tau} d\mathbf{y}\\
&=C^k \sum_{\substack{A\subset \{1,\ldots, k+1\}\\|A|\geq 1}}I_{A},
\end{aligned}
\]
where
\[
\begin{aligned}
I_{A}=\int_{[0,\infty)^k}\int_{\R^k}& \prod_{l\in A} \frac{1}{\eps^2} \1\{|\tau_l|\leq 100\eps^2, |y_l|\leq C\eps\} \prod_{l\in A^c}\frac{1}{\tau_l}e^{-\frac{y_l^2}{C\tau_l}}\1\{|\tau_l|>100\eps^2\}\\
&\1\{\sum_{l=1}^{k+1}\tau_l=t\}\1\{\sum_{l=1}^{k+1} y_l=x\} d\pmb{\tau} d\mathbf{y}.
\end{aligned}
\]
Here $A^c$ is the complement of $A$ in the set $\{1,\ldots, k+1\}$.

We seek to control the behavior as $\eps\to 0$ of all of these $I_{A}$ expressions. We consider the following two cases (as well as some subcases). Throughout, all bounds are assumed to hold for $\eps$ small enough, and the constants $C$ may change between lines.% \note{Yu: Please fix up the issue of whether the constants change between lines to make the final estimate work}.

\smallskip
\noindent{\bf Case 1:} $A^c=\varnothing$. We have
\[
\begin{aligned}
I_{A}=\int_{[0,\infty)^k}\int_{\R^k} \prod_{l=1}^{k} \frac{1}{\eps^2}\1\{|\tau_l|\leq 100\eps^2,|y_l|\leq C\eps\}
\frac{1}{\eps^2}\1\Big\{|t-\sum_{l=1}^k \tau_l|\leq 100\eps^2,|x-\sum_{l=1}^k y_l|\leq C\eps\Big\}d\pmb{\tau} d\mathbf{y}.
\end{aligned}
\]
For fixed $k$, the above integral equals to zero when $\eps$ is small enough to make $\sum_{l=1}^{k+1} 100\eps^2<t$. For arbitrary $k$, we have
\begin{equation}
I_{A}\leq C^k \eps^{k-2} \1\{k\geq 3\}.
\label{eq:bd2}
\end{equation}

\smallskip
\noindent{\bf Case 2:} $A^c\neq \varnothing$. Fix some $l^*\in A^c$ and define $\tilde{A}^c=A^c\setminus\{l^*\}$. We have
\[
\begin{aligned}
I_{A}=\int_{[0,\infty)^k}\int_{\R^k}& \prod_{l\in A} \frac{1}{\eps^2} \1\{|\tau_l|\leq 100\eps^2, |y_l|\leq C\eps\} \prod_{l\in \tilde{A}^c}\frac{1}{\tau_l}e^{-\frac{y_l^2}{C\tau_l}}\1\{|\tau_l|>100\eps^2\}\\
&\frac{1}{\tau_{l^*}}e^{-\frac{y_{l^*}^2}{C\tau_{l^*}}}\1\{|\tau_{l^*}|>100\eps^2\}\1\{\sum_{l=1}^{k+1}\tau_l=t\}\1\Big\{\sum_{l=1}^{k+1} y_l=x\Big\} d\pmb{\tau} d\mathbf{y}.
\end{aligned}
\]
By symmetry, we change variables and assume $l^*=k+1$, so $A\cup \tilde{A}^c=\{1,\ldots,k\}$, and
\[
\begin{aligned}
I_{A}=\int_{[0,\infty)^k} \int_{\R^k} &\prod_{l\in A} \frac{1}{\eps^2}\1\{|\tau_l|\leq 100\eps^2,|y_l|\leq C\eps\}\prod_{l\in \tilde{A}^c}\frac{1}{\tau_l} e^{-\frac{y_l^2}{C\tau_l}} \1\{|\tau_l|>100\eps^2\}\\
&\frac{1}{t-\sum_{l=1}^k \tau_l} e^{-\frac{(x-\sum_{l=1}^ky_l)^2}{C(t-\sum_{l=1}^k \tau_l)}}\1\Big\{\sum_{l=1}^k\tau_l<t\Big\}d\pmb{\tau} d\mathbf{y}.
\end{aligned}
\]

Integrating all $y_l$ for $l\in \tilde{A}^c$ yields
\begin{equation}
\begin{aligned}
I_{A}\leq C^{|\tilde{A}^c|}\int_{[0,\infty)^k}\int_{\R^{|A|}}&\prod_{l\in A} \frac{1}{\eps^2}\1\{|\tau_l|\leq 100\eps^2,|y_l|\leq C\eps\}\prod_{l\in\tilde{A}^c} \frac{1}{\sqrt{\tau_l}}\1\{|\tau_l|>100\eps^2\}\\
&\frac{1}{\sqrt{t-\sum_{l=1}^k \tau_l}}\frac{1}{\sqrt{t-\sum_{l\in A} \tau_l}} e^{-\frac{(x-\sum_{l\in A} y_l)^2}{C(t-\sum_{l\in A} \tau_l)}}\1\Big\{\sum_{l=1}^k\tau_l<t\Big\}d\pmb{\tau} d\mathbf{y}.
\end{aligned}
\label{eq:iab2}
\end{equation}
If $\tilde{A}^c\neq \varnothing$, we integrate $\tau_l$ for $l\in \tilde{A}^c$ and follow \eqref{eq:bd1} to obtain
\begin{equation}
\begin{aligned}
I_{A}\leq C^{|\tilde{A}^c|} e^{-\frac{|\tilde{A}^c|\log |\tilde{A}^c|}{C}}\int_{[0,\infty)^{|A|}}\int_{\R^{|A|}} &\prod_{l\in A} \frac{1}{\eps^2}\1\{|\tau_l|\leq 100\eps^2,|y_l|\leq C\eps\}
\\&\left(t-\sum_{l\in A} \tau_l\right)^{\frac{|\tilde{A}^c|}{2}-1}
\1\Big\{\sum_{l\in A}\tau_l<t\Big\}d\pmb{\tau} d\mathbf{y}.%\les \eps^{k(i)},
\end{aligned}
\label{eq:iab1}
\end{equation}
If $\tilde{A}^c=\varnothing$, then $A=\{1,\ldots,k\}$ and we integrate any $y_i$ in \eqref{eq:iab2} to obtain
\begin{equation}
\begin{aligned}
I_{A}\leq \int_{[0,\infty)^{k}}\int_{\R^{k-1}} &\frac{1}{\eps^2}\1\{|\tau_i|\leq 100\eps^2\}\prod_{l\neq i,l=1}^k \frac{1}{\eps^2}\1\{|\tau_l|\leq 100\eps^2,|y_l|\leq C\eps\}\\
&\frac{1}{\sqrt{t-\sum_{l=1}^k \tau_l}}
\1\Big\{\sum_{l=1}^k\tau_l<t\Big\}d\pmb{\tau} d\mathbf{y}.%\les \eps^{k(i)},
\end{aligned}
 \label{eq:iab}
 \end{equation}

We consider the following subcases of case 2.

\noindent{\bf Case 2.i:} $k\leq 3$. We can choose $\eps$ small in \eqref{eq:iab1} and \eqref{eq:iab} so that when $\tilde{A}^c\neq \varnothing$
\[
\begin{aligned}
I_{A}\leq C^k e^{-\frac{|\tilde{A}^c|\log |\tilde{A}^c|}{C}}t^{\frac{|\tilde{A}^c|}{2}-1}\eps^{|A|}\leq C^k t^{\frac{|\tilde{A}^c|}{2}-1} \eps^{|A|}.
\end{aligned}
\]
and when $\tilde{A}^c=\varnothing$,
\[
I_{A}\leq  C^k t^{-1} \eps^{|A|},
\]
Thus in both cases we have
\begin{equation}
I_{A}\leq C^kt^{\frac{|\tilde{A}^c|}{2}-1} \eps^{|A|}.
\label{eq:bd3}
\end{equation}

\noindent{\bf Case 2.ii:} $k>3$ and $|\tilde{A}^c|\geq 2$. We consider \eqref{eq:iab1} and the same discussion as in case 2.i leads to
\begin{equation}
I_{A}\leq C^k e^{-\frac{|\tilde{A}^c|\log |\tilde{A}^c|}{C}}t^{\frac{|\tilde{A}^c|}{2}-1}\eps^{|A|}.
\label{eq:bd4}
\end{equation}

\noindent{\bf Case 2.iii:} $k>3$ and $|\tilde{A}^c|=1$. We integrate any $\tau_{i}$ in \eqref{eq:iab1} to find
\[
\begin{aligned}
I_{A}\leq C^{|\tilde{A}^c|} e^{-\frac{|\tilde{A}^c|\log |\tilde{A}^c|}{C}}\int_{[0,\infty)^{|A|-1}}\int_{\R^{|A|}}& \frac{1}{\eps^2}\1\{|y_i|\leq C\eps\}\prod_{l\neq i,l\in A} \frac{1}{\eps^2}\1\{|\tau_l|\leq 100\eps^2,|y_l|\leq C\eps\}\\
&\Big(t-\sum_{l\neq i,l\in A} \tau_l\Big)^{1/2}\1\Big\{\sum_{l\neq i,l\in A}\tau_l<t\Big\}d\pmb{\tau} d\mathbf{y},
\end{aligned}
\]
thus,
\begin{equation}
I_{A}\leq C^ke^{-\frac{|\tilde{A}^c|\log |\tilde{A}^c|}{C}}\eps^{|A|-2}\sqrt{t}=C^k \sqrt{t}\eps^{k-3}.
\label{eq:bd5}
\end{equation}

\noindent{\bf Case 2.iv:} $k>3$ and $\tilde{A}^c=\varnothing$. We integrate any $\tau_i$ in \eqref{eq:iab} to find
\[
\begin{aligned}
I_{A}\leq \int_{[0,\infty)^{k-1}}\int_{\R^{k-1}} \frac{1}{\eps^2}\prod_{l\neq i,l=1}^k \frac{1}{\eps^2}\1\{|\tau_l|\leq 100\eps^2,|y_l|\leq C\eps\}
\Big(t-\sum_{l\neq i,l=1}^k \tau_l\Big)^{1/2}\1\Big\{\sum_{l\neq i,l=1}^k\tau_l<t\Big\}d\pmb{\tau} d\mathbf{y},
\end{aligned}
\]
thus
\begin{equation}
I_{A}\leq C^k\sqrt{t} \eps^{k-3}.
\label{eq:bd51}
\end{equation}

\medskip
To summarize, by \eqref{eq:bd2},\eqref{eq:bd3},\eqref{eq:bd4}, \eqref{eq:bd5} and \eqref{eq:bd51}, the following estimate holds for all $k$:
\begin{equation}
\begin{aligned}
&\left\|\frac{1}{\eps^{1+k}} e^{\frac{t_\eps}{\eps^2}\Ione(v)+\frac{x_\eps}{\eps}\Itwo(v)}\psi_{\eps,k}\right\|_{L^2(\Delta_{k,\eps}(t)\times\R^k)}^2\\
&\leq C^k \left( \eps^{k-2}\1\{k\geq 3\}+\sum_{a=1}^k {k+1\choose a} \mathcal{E}(k,a)\right),
\end{aligned}
\label{eq:allbd}
\end{equation}
where $a^c=k+1-a$, $\tilde{a}^c=a^c-1=k-a$ (so that $a=|A|$, $a^c=|A^c|$ and $\tilde{a}^c=|\tilde{A}^c|$), and
\[
\begin{aligned}
\mathcal{E}(k,a)
=\1\{k\leq 3\}t^{\frac{\tilde{a}^c}{2}-1} \eps^{a}+\1\{k>3,\tilde{a}^c\geq 2\}e^{-\frac{\tilde{a}^c\log \tilde{a}^c}{C}}t^{\frac{\tilde{a}^c}{2}-1}\eps^{a}+\1\{k> 3,\tilde{a}^c\leq 1\} \sqrt{t}\eps^{k-3}.
\end{aligned}
\]
It is clear that for each fixed $k$, the right-hand side of \eqref{eq:allbd} goes to zero as $\eps\to 0$ since $a\geq 1$. In light of the convergence in \eqref{eq:dct}, the proof of \eqref{eq:conco} is complete.

\vskip 5mm
In the end we provide a uniform estimate in $k$ when $k$ is large, and from now on we fix the large constant $C$ in \eqref{eq:allbd}. Considering the terms $\1\{k\leq 3\}t^{\frac{\tilde{a}^c}{2}-1} \eps^{a}$ and $\1\{k> 3,\tilde{a}^c\leq 1\} \sqrt{t}\eps^{k-3}$ from the above expression for $\mathcal{E}(k,a)$, we have
\begin{equation}
\begin{aligned}
&\sum_{a=1}^{k} {k+1\choose a}\left(\1\{k\leq 3\}t^{\frac{\tilde{a}^c}{2}-1} \eps^{a}+\1\{k> 3,\tilde{a}^c\leq 1\} \sqrt{t}\eps^{k-3}\right)\\
&\leq \1\{k\leq 3\}6\left(1+\frac{1}{\sqrt{t}}+\frac{1}{t}\right)\eps+\1\{k>3\}\left[(k+1)+\frac{(k+1)k}{2}\right] \sqrt{t}\eps^{k-3}.
\end{aligned}
\label{eq:bd6}
\end{equation}

For $\1\{k>3,\tilde{a}^c\geq 2\}e^{-\frac{\tilde{a}^c\log \tilde{a}^c}{C}}t^{\frac{\tilde{a}^c}{2}-1}\eps^{a}$, the remaining term in $\mathcal{E}(k,a)$, we write
\[
e^{-\frac{\tilde{a}^c\log \tilde{a}^c}{C}}t^{\frac{\tilde{a}^c}{2}-1}\eps^{a}=\frac{1}{t} \left(\frac{\sqrt{t}}{(\tilde{a}^c)^{\frac{1}{C}}}\right)^{\tilde{a}^c} \eps^{a},
\]
and we choose $M$ sufficiently large (only depending on $C,t$) so that when $\tilde{a}^c>M$,
\[
\frac{\sqrt{t}}{(\tilde{a}^c)^{\frac{1}{C}}}+\eps<\frac{\sqrt{t}}{M^{\frac{1}{C}}}+\eps<\frac{1}{C}.
\]
Thus,
\[
\begin{aligned}
e^{-\frac{\tilde{a}^c\log \tilde{a}^c}{C}}t^{\frac{\tilde{a}^c}{2}-1}\eps^{a}
\leq \1\{\tilde{a}^c>M\}\frac{1}{t} \left(\frac{\sqrt{t}}{M^{\frac{1}{C}}}\right)^{\tilde{a}^c} \eps^{a}
+\1\{\tilde{a}^c\leq M\}(1+t^{\frac{M}{2}-1})\eps^{k-M}.
\end{aligned}
\]
Thus, when $k>M$, we have
\begin{equation}
\begin{aligned}
&\sum_{a=1}^{k}  {k+1\choose a}\1\{\tilde{a}^c\geq 2\}e^{-\frac{\tilde{a}^c\log \tilde{a}^c}{C}}t^{\frac{\tilde{a}^c}{2}-1}\eps^{a}\\
\leq &\left[\sum_{a=1}^{k}  {k+1\choose a}\frac{1}{t} \left(\frac{\sqrt{t}}{M^{\frac{1}{C}}}\right)^{\tilde{a}^c} \eps^{a}+\sum_{a=1}^{k}  {k+1\choose a}(1+t^{\frac{M}{2}-1})\eps^{k-M}\right]\\
\leq & \left[\frac{M^{\frac{1}{C}}}{t^\frac32} (\frac{\sqrt{t}}{M^{\frac{1}{C}}}+\eps)^{k+1}+2^{k+1}(1+t^{\frac{M}{2}-1})\eps^{k-M}\right].
\end{aligned}
\label{eq:bd7}
\end{equation}

Combining \eqref{eq:allbd},\eqref{eq:bd6} and \eqref{eq:bd7}, we conclude that when $k>M$,  %\note{Note that the convergence of the below terms requires the comparison of certain terms to $1/C$}
\[
\begin{aligned}
&\left\|\frac{1}{\eps^{1+k}} e^{\frac{t_\eps}{\eps^2}\Ione(v)+\frac{x_\eps}{\eps}\Itwo(v)}\psi_{\eps,k}\right\|_{L^2(\Delta_{k,\eps}(t)\times\R^k)}^2\\
\leq & C^k \eps^{k-2}+C^k\left[(k+1)+\frac{(k+1)k}{2}\right] \sqrt{t}\eps^{k-3}
+C^k\left[\frac{M^{\frac{1}{C}}}{t^\frac32} (\frac{\sqrt{t}}{M^{\frac{1}{C}}}+\eps)^{k+1}+2^{k+1}(1+t^{\frac{M}{2}-1})\eps^{k-M}\right].
\end{aligned}
\]
The right-hand side is summable in $k>M$ when $\eps\ll 1$. Recalling the estimates in \eqref{eq:bd1} and \eqref{eq:bd11}, we complete the proof of \eqref{eq:ubdl2}.
\end{proof}

\section{Exactly solvable Beta polymer: moment convergence}\label{s:beta}

In this section, we study the limit of the exact moment formulas for the Beta polymer under the scaling of Theorem~\ref{c:poly}. Such moments formulas were first found in the work of \cite{barraquand2015random}, though the point-to-point moments we consider here are given explicitly in \cite{LeDoussal-Thierry}. As $\eps\to 0$, we prove on the formal level that these formulas converge to the corresponding SHE moment formulas given in \cite[Section 6.2]{BigMac}. We will not provide herein a rigorous proof of these moment formula asymptotics (similar asymptotics are present, for instance, in \cite{BigMac}) but rather just work at the level of critical point analysis of the contour integrals. This provides an independent confirmation of the correctness of Theorem~\ref{c:poly} (though even a rigorous proof of these moment formula convergences would not imply Theorem~\ref{c:poly} due to the lack of well-posedness of the moment problem for the SHE).

Recall that for the Beta polymer, $B_{i,j}$ has $\mathrm{Beta}(\alpha,\beta)$ distribution. Let $\mu=\alpha$ and $\nu=\alpha+\beta$, then it follows from \cite{barraquand2015random,LeDoussal-Thierry} that for $T\in \Z_{\geq 0}$ and $n_1\geq \ldots\geq n_k\in \Z$,
\begin{equation}
\begin{aligned}
&\E\{Z(T,n_1)\ldots Z(T,n_k)\}\\
&=\frac{(\nu)_k}{(2\pi i)^k}\int\ldots\int \prod_{1\leq A<B\leq k} \frac{z_A-z_B}{z_A-z_B-1}\prod_{j=1}^k \left(\frac{\nu+z_j}{z_j}\right)^{n_j}\left(\frac{\mu+z_j}{\nu+z_j}\right)^T \frac{dz_j}{(\nu+z_j)^2},
\end{aligned}
\label{eq:mmpa}
\end{equation}
where the contour for $z_k$ is a small circle around the origin, and the contour for $z_A$ contains the contour for $z_{B}+1$ for all $1\leq A<B\leq k$, as well as the origin, but all contours exclude $-\nu$. Here $(\nu)_k=\nu(\nu+1)\ldots (\nu+k-1)$.

There is a similar equation given in \cite[Section 6.2]{BigMac} for the moments of the SHE. In terms of the scalings of Theorem~\ref{c:poly}, it implies that for $t\geq 0$ and $x_1\geq  \cdots\geq x_k$,
\begin{equation}
\begin{aligned}
&\E\{ \V(t,x_1)\ldots \V(t,x_k)\}\\
&=\frac{(1-2\gamma)^{2k}}{(2\pi i)^k 2^k(1-\gamma)^{2k}} \int\ldots \int \prod_{1\leq A<B\leq k} \frac{z_A-z_B}{z_A-z_B-1}\prod_{j=1}^k \exp\left(\frac{(1-2\gamma)^4tz_j^2}{8\gamma(1-\gamma)}-\frac{(1-2\gamma)^2x_jz_j}{2\gamma(1-\gamma)}\right)dz_j,
\end{aligned}
\label{eq:mmshe}
\end{equation}
where the contour for $z_j$ is $r_j+i\R$ for arbitrary real $r_1,\ldots, r_k$ such that $r_j>r_{j+1}+1$ for all $j$.

The goal is to prove the convergence of \eqref{eq:mmpa} to \eqref{eq:mmshe} after a proper rescaling, and as we stressed at the beginning, we will not provide a detailed rigorous proof but only sketch the critical point analysis. We fix $t>0,x_j\in\R$ and $\gamma\in(0,\frac12)$, and let $\alpha=\beta=\eps^{-1}$,
%we have
%\[
%\E\left\{\prod_{j=1}^2 \frac{1}{\eps}e^{\frac{t_\eps}{\eps^2}\Ione(1-2\gamma)-\frac{2x_{\eps,j}}{\eps}\Itwo(1-2\gamma)}Z\left(\frac{t_\eps}{\eps^2},\gamma\frac{t_\eps}{\eps^2}+\frac{x_{\eps,j}}{\eps}\right)\right\}\to \E\{\V(t,x_1)\V(t,x_2)\}
%\]
%as $\eps\to0$, with $\V$ solving \eqref{eq:spde1}.
and (recalling $(t_{\eps},x_{\eps})$ defined before Theorem \ref{t:mainTh})
$
T=\frac{t_\eps}{\eps^2}$, 
$
n_j=\gamma\frac{t_\eps}{\eps^2}+\frac{x_{\eps,j}}{\eps}.
$
%the goal is to derive the limit of
%\[
%\begin{aligned}
%&\E\{Z(T,n_1)Z(T,n_2)\}\\
%&=\frac{2\alpha(2\alpha+1)}{(2\pi i)^2}\int\int \frac{z_1-z_2}{z_1-z_2-1} \prod_{j=1}^2 \left(\frac{2\alpha+z_j}{z_j}\right)^{n_j} \left(\frac{\alpha+z_j}{2\alpha+z_j}\right)^T \frac{dz_j}{(2\alpha+z_j)^2}
%\end{aligned}
%\]
%after a proper rescaling.
We define
%\begin{equation}
%f_{\eps,j}(z)=\log \frac{\alpha+z}{2\alpha+z}+\frac{n_j}{T}\log\frac{2\alpha+z}{z},
%\end{equation}
\[
f_\eps(z,n)=\log \frac{\alpha+z}{2\alpha+z}+\frac{n}{T}\log\frac{2\alpha+z}{z},\qquad \textrm{so}\qquad 
\left(\frac{2\alpha +z_j}{z_j}\right)^{n_j} \left(\frac{\alpha+z_j}{2\alpha+z_j}\right)^T= e^{Tf_\eps(z_j,n_j)}%e^{T[f_{\eps}(z_1,n_1)+f_{\eps}(z_2,n_2)]}.
\]
A simple calculation shows that the critical point of $f_{\eps}$ is given by
\begin{equation}
\frac{2n}{T-2n}\alpha\sim \frac{2\gamma}{(1-2\gamma)}\eps^{-1}=:z_{0,\eps}.%+\frac{2x_j}{(1-2\gamma)^2t}\eps\alpha+O(\eps^2\alpha).
\end{equation}
We may deform our contours to lie close to the critical point $z_{0,\eps}$. Since $\eps\ll 1$, in the vicinity of $z_{0,\eps}$ the contours can be approximated by vertical straight lines of length on the order of $\eps^{-1}$. We will assume (without proof) that with small error, we can replace the integrand by its approximation around the critical point (such an argument would involve describing steep-descent contours and similar examples can be found, for instance, in \cite{BigMac}). Taylor expanding to second order around $z_{0,\eps}$ and setting $z_j=z_{0,\eps}+\tilde{z}_j$, we find
\[
%f_\eps(z_j,n_j)\approx f_\eps(z_{0,\eps},n_j)+f'_\eps(z_{0,\eps},n_j)\tilde{z}_j+\frac12f''_\eps(z_{0,\eps},n_j)\tilde{z}_j^2,
\begin{aligned}
f_\eps(z_j,n_j)\approx  
%& f_\eps(z_{0,\eps},n_j)+f'_\eps(z_{0,\eps},n_j)\tilde{z}_j+\frac12f''_\eps(z_{0,\eps},n_j)\tilde{z}_j^2\\
%=&\gamma\log \frac{1-\gamma}{\gamma}-\log(2-2\gamma)+\frac{\eps x_{\eps,j}}{t_\eps} \log \frac{1-\gamma}{\gamma}\\
%&-\frac{(1-2\gamma)^2}{2\gamma(1-\gamma)}\frac{\eps^2 x_{\eps,j}}{t_\eps}\tilde{z}_j+\frac12\left(\frac{\eps^2(1-2\gamma)^4}{4\gamma(1-\gamma)}+\frac{(1-2\gamma)^3}{4\gamma^2(1-\gamma)^2}\frac{\eps^3 x_{\eps,j}}{t_\eps}\right)\tilde{z}_j^2\\
%=&
-\I(1-2\gamma)+2\I'(1-2\gamma)\frac{\eps x_{\eps,j}}{t_\eps}+\frac{(1-2\gamma)^4\eps^2}{8\gamma(1-\gamma)}\tilde{z}_j^2-\frac{(1-2\gamma)^2\eps^2 x_{\eps,j}}{2\gamma(1-\gamma)t_\eps}\tilde{z}_j+O(\eps^3),
\end{aligned}
\]
hence (under the aforementioned critical point hypothesis) we find that
%\[
%\begin{aligned}
%&\E\{Z(T,n_1)\ldots Z(T,n_k)\}\\
%&\approx\frac{(2\alpha)_k}{(2\pi i)^k}\int \cdots\int  \prod_{1\leq A<B\leq k} \frac{\tilde{z}_A-\tilde{z}_B}{\tilde{z}_A-\tilde{z}_B-1}\prod_{j=1}^ke^{T[f_\eps(z_{0,\eps},n_j)+f'_\eps(z_{0,\eps},n_j)\tilde{z}_j+\frac12f''_\eps(z_{0,\eps},n_j)\tilde{z}_j^2]}
% \frac{d\tilde{z}_j}{(2\alpha+z_{0,\eps}+\tilde{z}_j)^2}.
%\end{aligned}
%\]
%Now we pass to the limit after a rescaling to obtain
\[
\begin{aligned}
&\E\left\{\prod_{j=1}^k \frac{1}{\eps}e^{\frac{t_\eps}{\eps^2}\Ione(1-2\gamma)-\frac{2x_{\eps,j}}{\eps}\Itwo(1-2\gamma)}Z\left(\frac{t_\eps}{\eps^2},\gamma\frac{t_\eps}{\eps^2}+\frac{x_{\eps,j}}{\eps}\right)\right\}\\
&\to\frac{(1-2\gamma)^{2k}}{(2\pi i)^k2^k(1-\gamma)^{2k}} \int \cdots\int \prod_{1\leq A<B\leq k} \frac{\tilde{z}_A-\tilde{z}_B}{\tilde{z}_A-\tilde{z}_B-1}\prod_{j=1}^k \exp\left(\frac{(1-2\gamma)^4t\tilde{z}_j^2}{8\gamma(1-\gamma)}-\frac{(1-2\gamma)^2x_j\tilde{z}_j}{2\gamma(1-\gamma)}\right)d\tilde{z}_j,
\end{aligned}
\]
where the contours are as in \eqref{eq:mmshe}. The right-hand side equals $\E\{\V(t,x_1)\ldots\V(t,x_k)\}$ as desired.

\appendix

\section{Sharp large deviation for symmetric simple random walk}\label{app:A}

Recall that
\[
\cP_\eps(t_\eps,x_\eps,m_1,m_2)= \PP\left(S_{\frac{t_\eps}{\eps^2}+m_1}=v\frac{t_{\eps}}{\eps^2}+\frac{x_\eps}{\eps}+m_2\right),
\]
$p_{\sigma^2}(t,x)$ is the density of $N(0,\sigma^2t)$, and assume that $m_1,m_2\in \Z$, $|m_1|,|m_2|\leq 1$ and $m_1-m_2$ is even.
 \begin{lem}
 \label{l:ldprw}
 For fixed $t>0,x\in\R$, we have
\begin{equation}
\begin{aligned}
\frac{1}{\eps}e^{\frac{t_\eps}{\eps^2}\Ione(v)+\frac{x_\eps}{\eps}\Itwo(v)}\cP_\eps(t_\eps,x_\eps,m_1,m_2)\to  \frac{2p_{1-v^2}(t,x)}{(1+v)^{\frac{m_1+m_2}{2}}(1-v)^{\frac{m_1-m_2}{2}}}
\end{aligned}
\label{eq:pcldp}
\end{equation}
 as $\eps\to 0$, and there exists a constant $C>0$ such that for all $t>0,x\in \R$,
\begin{equation}
\begin{aligned}
\frac{1}{\eps}e^{\frac{t_\eps}{\eps^2}\Ione(v)+\frac{x_\eps}{\eps}\Itwo(v)}\cP_\eps(t_\eps,x_\eps,m_1,m_2)
\leq C\left(\frac{1}{\eps}\1\{t\leq 100\eps^2,|x|\leq C\eps\}+\1\{t>100\eps^2\}\frac{1}{\sqrt{t}}e^{-\frac{x^2}{Ct}}\right).
\end{aligned}
\label{eq:uesrw}
\end{equation}
\end{lem}

\begin{proof}
To simplify the notation, let
\begin{equation}
n=\frac{t_\eps}{\eps^2}+m_1, \ \
m=v\frac{t_\eps}{\eps^2}+\frac{x_\eps}{\eps}+m_2,
\label{eq:mnij}
\end{equation}
so that $\cP_\eps(t_\eps,x_\eps,m_1,m_2)=\PP(S_n=m)$. We will utilize the notation $\les$ when the left-hand side is bounded by a constant (independent of $\eps$ and $t$ and $x$) times the right-hand side for all $\eps$ sufficiently small. We first prove the estimate \eqref{eq:uesrw} that is uniform in $t>0,x\in\R$, then prove the convergence in \eqref{eq:pcldp} for fixed $t>0,x\in\R$.

\smallskip
\noindent{\bf Case 1:} $t\leq 100\eps^2$. It is clear that $t_\eps\leq 100\eps^2$, and by the fact that $|m|\leq n$ (or else $\PP(S_n=m)=0$), there exists $C>0$ so that $|x_\eps|\leq C\eps$. Thus, we have
\[
\frac{1}{\eps}e^{\frac{t_\eps}{\eps^2}\Ione(v)+\frac{x_\eps}{\eps}\Itwo(v)}\cP_\eps(t_\eps,x_\eps,m_1,m_2)\les\frac{1}{\eps} \1\{|x_\eps|\leq C\eps\}.
\]
This proves the first bound on the right-hand side of \eqref{eq:uesrw}.

\noindent{\bf Case 2:} $t>100\eps^2$. Recall that $|m_1|,|m_2|\leq 1$, so $n\geq 90$ -- we will implicitly use the largeness of $n$ in some of the bounds below. By \eqref{eq:mnij},
 \[
 \begin{aligned}
 \frac{t_\eps}{\eps^2}\Ione(v)+\frac{x_\eps}{\eps}\Itwo(v)=\frac12 n\log(1-v^2)+\frac12m\log\frac{1+v}{1-v}
 -m_1\Ione(v)+(vm_1-m_2)\Itwo(v),
\end{aligned}
  \]
  which implies
  \[
  e^{\frac{t_\eps}{\eps^2}\Ione(v)+\frac{x_\eps}{\eps}\Itwo(v)}\les e^{\frac12 n\log(1-v^2)+\frac12m\log\frac{1+v}{1-v}}.
  \]
Using Stirling's approximation,
  \[
  \sqrt{2\pi} N^{N+\frac12}e^{-N}\leq N!\leq eN^{N+\frac12}e^{-N},
  \]
we derive the bound
  \[
  \begin{aligned}
  \PP(S_n=m)\les \sqrt{\frac{n}{n^2-m^2}} e^{-\frac{n+m}{2}\log(1+\frac{m}{n})-\frac{n-m}{2}\log(1-\frac{m}{n})}
  \end{aligned}
  \]
when $|m|<n$.

Now we consider three different subcases of case 2. Fix $0<\tau\ll 1$, and let $k=m/n$.

\noindent{\bf Case 2.i:} $|k|=1$. We have $P(S_n=m)=2^{-n}$, so
\[
\begin{aligned}
\frac{1}{\eps}e^{\frac{t_\eps}{\eps^2}\Ione(v)+\frac{x_\eps}{\eps}\Itwo(v)}\cP_\eps(t_\eps,x_\eps,m_1,m_2)
\les \frac{1}{\eps}e^{\frac12 n\log(1-v^2)+\frac12m\log\frac{1+v}{1-v}-n\log 2}.
\end{aligned}
\]
In both cases of $m=n$ and $m=-n$, using the fact that $v\in(0,1)$, we have
\[
\frac{1}{\eps}e^{\frac12 n\log(1-v^2)+\frac12m\log\frac{1+v}{1-v}-n\log 2}\leq \frac{1}{\eps} e^{-\delta n}
\]
for some $\delta>0$ depending on $v$. Furthermore,
\[
\frac{1}{\eps}e^{-\delta n}=\frac{1}{\sqrt{\eps^2 n}}\sqrt{n}e^{-\frac{\delta n}{2}}e^{-\frac{\delta n}{2}} \les \frac{1}{\sqrt{t}}e^{-\frac{\delta n}{2}},
\]
so it remains to show that $x^2/t\les n$. By the fact that $|m|\leq n$, we have
\[
-t_\eps\les\eps x_\eps\les t_\eps,
\]
which implies $-t\les \eps x\les t$. This implies that in this case, the second bound on the right-hand side of \eqref{eq:uesrw} holds.

\noindent{\bf Case 2.ii:} $|k|<1$ and $|k-v|\geq \tau$. We have
\begin{equation}
\begin{aligned}
\frac{1}{\eps}e^{\frac{t_\eps}{\eps^2}\Ione(v)+\frac{x_\eps}{\eps}\Itwo(v)}\cP_\eps(t_\eps,x_\eps,m_1,m_2) \les \frac{1}{\eps} \sqrt{\frac{n}{n^2-m^2}}e^{-\frac{n}{2}F(k)},
\end{aligned}
\label{eq:pfuni}
\end{equation}
where
\begin{equation}
F(k)=(1+k)\log \frac{1+k}{1+v}+(1-k)\log \frac{1-k}{1-v}, \ \ k \in (-1,1).
\label{eq:defF}
\end{equation}
It is straightforward to check that $F(k)$ attains its minimum at $v$ and $F''(k)$ is bounded from below by some positive constant. Since $|k-v|\geq \tau$, we have $F(k)\geq \delta\tau^2$ with $\delta=\frac12\min_{k\in(-1,1)} F''(k)$. In addition, since $|m|<n$, we have $-n+2\leq m\leq n-2$, so
\[
\frac{1}{\eps} \sqrt{\frac{n}{n^2-m^2}} \leq \frac{1}{\sqrt{\eps^2 n}}\frac{1}{\sqrt{1-\left(\frac{m}{n}\right)^2}} \les \frac{1}{\sqrt{t}}\sqrt{n}.
\]
Therefore, we have
\[
\begin{aligned}
\frac{1}{\eps}e^{\frac{t_\eps}{\eps^2}\Ione(v)+\frac{x_\eps}{\eps}\Itwo(v)}\cP_\eps(t_\eps,x_\eps,m_1,m_2)
  \les \frac{1}{\sqrt{t}} \sqrt{n}e^{-\frac{ n}{2}\delta\tau^2}\les \frac{1}{\sqrt{t}} e^{-\frac{n}{4}\delta\tau^2},
\end{aligned}
\]
and the same discussion as in case 2.i shows $e^{-\frac{n}{4}\delta\tau^2}\les e^{-\frac{x^2}{Ct}}$ for some $C>0$, matching the second bound on the right-hand side of \eqref{eq:uesrw}.

\noindent{\bf Case 2.iii:} $|k-v|<\tau$. We have
\[
\frac{1}{\eps} \sqrt{\frac{n}{n^2-m^2}}\les \frac{1}{\sqrt{t}}.
\]
For the exponent, it is clear that $F(k)\geq \delta (k-v)^2$ with the same $\delta>0$ from case 2.ii, so
\[
e^{-\frac{n}{2}F(k)}\leq e^{-\frac{n}{2}\delta (k-v)^2}.
\]
%Given \eqref{eq:pfuni}, we only need to show
%\[
%\frac{x^2}{t}\les n(k-c_1)^2.
%\]
Since $k=m/n$ with $m,n$ given in \eqref{eq:mnij}, we have
\begin{equation}
n(k-v)^2=\frac{(x_\eps+(m_2-vm_1)\eps)^2}{t_\eps+m_1\eps^2},
\label{eq:kv}
\end{equation}
so
\[
 e^{-\frac{n}{2}\delta (k-v)^2}\les  e^{-\frac{(|x|-\tilde{C}\eps)^2}{Ct}}\leq e^{-\frac{x^2}{Ct}}e^{\frac{2\tilde{C}\eps |x|}{Ct}}
 \]
 for some $C,\tilde{C}>0$. For $M>0$, if $|x|\leq M\eps$, we have the desired estimate; if $|x|>M\eps$, we have
 \[
 e^{-\frac{x^2}{Ct}}e^{\frac{2\tilde{C}\eps |x|}{Ct}}<  e^{-\frac{x^2}{Ct}}e^{\frac{2\tilde{C}x^2}{MCt}},
 \]
 so we only need to choose $M$ sufficiently large to complete the proof of \eqref{eq:uesrw}.

 To prove \eqref{eq:pcldp}, we note that for fixed $t>0,x\in\R$ and sufficiently small $\eps$, $|k-v|\ll 1$ (we are in the region of case 2.iii). We use Stirling's approximation and the fact that $n,\frac{n+m}{2},\frac{n-m}{2}\to \infty$ to obtain that
\[
\frac{\PP(S_n=m)}{\sqrt{\frac{2n}{\pi(n^2-m^2)}}e^{-\frac{n}{2}[(1+k)\log(1+k)+(1-k)\log(1-k)]}}\to 1
\]
as $\eps\to 0$. Thus we only need to analyze
\[
\begin{aligned}
&\frac{1}{\eps}e^{\frac{t_\eps}{\eps^2}\Ione(v)+\frac{x_\eps}{\eps}\Itwo(v)}\sqrt{\frac{2n}{\pi(n^2-m^2)}}e^{-\frac{n}{2}[(1+k)\log(1+k)+(1-k)\log(1-k)]}\\
&=\sqrt{\frac{2n}{\pi(n^2-m^2)\eps^2}} \frac{e^{-\frac{n}{2}F(k)}}{(1+v)^{\frac{m_1+m_2}{2}}(1-v)^{\frac{m_1-m_2}{2}}},
\end{aligned}
\]
with $F(k)$ defined in \eqref{eq:defF}. First, we have
\[
\sqrt{\frac{2n}{\pi(n^2-m^2)\eps^2}}\to \sqrt{\frac{2}{\pi(1-v^2)t}}.
\]
Secondly, by \eqref{eq:kv}, we have $|k-v|\les\eps$ for fixed $t>0,x\in\R$ and $n(k-v)^2\to \frac{x^2}{t}$
as $\eps\to0$. We expand $F(k)=\frac12F''(v)(k-v)^2+O(|k-v|^3)$ and conclude that
\[
\frac{n}{2}F(k)\to F''(v)\frac{x^2}{4t}=\frac{x^2}{2(1-v^2)t}
\]
as $\eps\to0$. This completes the proof of \eqref{eq:pcldp}.
\end{proof}

%\bibliographystyle{abbrv}
%\bibliography{bibliography}

\begin{thebibliography}{1}

\bibitem{alberts2014intermediate}
T.~Alberts, K.~Khanin, J.~Quastel.
\newblock The intermediate disorder regime for directed polymers in dimension $1+1$.
\newblock {\em Ann. Probab.}, 42:1212--1256, 2014.

\bibitem{ACQ}
G.~Amir, I.~Corwin, J.~Quastel.
\newblock Probability distribution of the free energy of the continuum directed random polymer in $1+1$ dimensions.
\newblock {\em Commun. Pure Appl. Math.}, {\bf 64}:466--537, 2011.

\bibitem{balazs2006random}
M.~Bal{\'a}zs, F.~Rassoul-Agha, T.~Sepp{\"a}l{\"a}inen.
\newblock The random average process and random walk in a space-time random   environment in one dimension.
\newblock {\em Commun. Math. Phys.}, 266:499--545, 2006.

\bibitem{barraquand2015random}
G.~Barraquand and I.~Corwin.
\newblock Random-walk in beta-distributed random environment.
\newblock  {\em Probab. Theo. Rel. Fields}, onlinefirst, 2016.

%\bibitem{barraquand2016random}
%G.~Barraquand and I.~Corwin.
%\newblock The beta rwre and bernoulli-exponential fpp.
%\newblock {\em in preparation}, 2016.

\bibitem{BigMac}
A. Borodin, I.~Corwin.
\newblock Macdonald processes.
\newblock {\em Probab. Theo. Rel. Fields}, {\bf 158}:225--400, 2014.

\bibitem{caravenna2013polynomial}
F.~Caravenna, R.~Sun, and N.~Zygouras.
\newblock Polynomial chaos and scaling limits of disordered systems.
\newblock {\em J. Eur. Math. Soc.}, to appear.

\bibitem{ICReview}
I.~Corwin.
\newblock The {K}ardar-{P}arisi-{Z}hang equation and universality class.
\newblock {\em Rand. Mat.: Theo. Appl.}, {\bf 1}:1130001, 2012.

\bibitem{corwin2014macdonald}
I.~Corwin.
\newblock Macdonald processes, quantum integrable systems and the Kardar-Parisi-Zhang universality class.
\newblock {\em Proceedings of the International Congress of Mathematicians 2014}.

\bibitem{LeDoussal-Thierry}
T.Thiery, P. Le Doussal.
\newblock Exact solution for a random walk in a time-dependent 1D random environment: the point-to-point Beta polymer
\newblock {\em arXiv preprint arXiv:1605.07538}, 2016.

\bibitem{RASY13}
F.~Rassoul-Agha, T.~Sepp\"{a}l\"{a}inen, A.~Yilmaz.
\newblock Quenched free energy and large deviations for random walks in random potentials.
\newblock {\em Commun. Pure Appl. Math.}, {\bf 66}:202--244, 2013.
\end{thebibliography}

\end{document}